\documentclass{amsart}
\usepackage[utf8]{inputenc}
\usepackage{amsmath, amssymb}

\usepackage{xcolor}
\usepackage[colorlinks,
    linkcolor={red!50!black},
    citecolor={blue!50!black},
    urlcolor={blue!80!black}]{hyperref}

\usepackage{tikz}
\usetikzlibrary{calc, matrix, arrows, cd}
\usetikzlibrary{decorations.markings}

\theoremstyle{plain}
\newtheorem{proposition}{Proposition}[section]
\newtheorem{theorem}[proposition]{Theorem}
\newtheorem{corollary}[proposition]{Corollary}
\newtheorem{lemma}[proposition]{Lemma}

\newtheorem{innercustomthm}{Theorem}
\newenvironment{customthm}[1]
  {\renewcommand\theinnercustomthm{#1}\innercustomthm}
  {\endinnercustomthm}

\theoremstyle{definition}
\newtheorem{definition}[proposition]{Definition}
\newtheorem{question}[proposition]{Question}

\newtheorem{construction}[proposition]{Construction}

\theoremstyle{remark}
\newtheorem{remark}[proposition]{Remark}

\newtheorem{example}[proposition]{Example}
\newtheorem*{claim}{Claim}
\newtheorem*{notation}{Notation}

\newcommand{\lk}{\operatorname{lk}}
\newcommand{\sign}{\operatorname{sign}}

\newcommand{\wt}{\widetilde}
\newcommand{\wh}{\widehat}
\newcommand{\sm}{\setminus}

\newcommand{\Q}{\mathbb{Q}}
\newcommand{\Z}{\mathbb{Z}}

\newcommand{\C}{\mathbb{C}}
\newcommand{\R}{\mathbb{R}}

\newcommand{\Arf}{\operatorname{Arf}}

\newcommand{\op}[1]{\operatorname{#1}}
\newcommand{\ol}[1]{\overline{#1}}
\newcommand{\dsign}{\operatorname{dsign}}
\newcommand{\pt}{\op{pt}}
\newcommand{\csum}{\#\,}
\newcommand{\sn}{\op{sn}}
\newcommand{\ul}{\underline}
\newcommand{\sarr}{\xrightarrow{\sim}}
\newcommand{\Int}{\op{Int}}
\newcommand{\id}{\op{id}}
\newcommand{\bsm}{\left(\begin{smallmatrix}}
\newcommand{\esm}{\end{smallmatrix}\right)}

\begin{document}
\title{Stably slice disks of links}

\author{Anthony Conway}
\address{Department of Mathematics, Durham University, United Kingdom}
\email{anthonyyconway@gmail.com}

\author{Matthias Nagel}
\address{Department of Mathematics, University of Oxford, United Kingdom}
\email{nagel@maths.ox.ac.uk}

\date{\today}

\begin{abstract}
We define the stabilizing number $\operatorname{sn}(K)$ of a knot $K \subset S^3$ as the minimal number $n$ of $S^2 \times S^2$ connected summands required for $K$ to bound a nullhomologous locally flat disk in $D^4 \#\, n S^2 \times S^2$. This quantity is defined when the Arf invariant of $K$ is zero.
 We show that $\operatorname{sn}(K)$ is bounded below by signatures and Casson-Gordon invariants and bounded above by the topological $4$-genus~$g_4^{\operatorname{top}}(K)$.
We provide an infinite family of examples with~$\operatorname{sn}(K)<g_4^{\operatorname{top}}(K)$.
\end{abstract}
\maketitle

\section{Introduction}
Several questions in $4$--dimensional topology simplify considerably after connected summing with sufficiently many copies of $S^2\times S^2$. 
For instance, Wall showed that homotopy equivalent, 
 simply-connected smooth $4$--manifolds become diffeomorphic after enough such \emph{stabilizations} \cite{Wall}. 
While other striking illustrations of this phenomenon can be found in \cite{FreedmanKirby, Quinn, Bohr, Auckly15, JuhaszZemke}, this paper focuses on embeddings of disks in stabilized $4$--manifolds.

A link $L \subset S^3$ is \emph{stably slice} if there exists $n \geq 0$ such that the components of~$L$ bound a collection of disjoint locally flat nullhomologous disks in the manifold~$D^4 \csum n S^2 \times S^2.$ 
The \emph{stabilizing number} of a stably slice link is defined as 
\[ \operatorname{sn}(L):= \operatorname{min} \lbrace n \ | \ L \text{ is slice in } D^4 \csum n S^2 \times S^2 \rbrace. \]

Stably slice links have been characterized by Schneiderman~\cite[Theorem~1, Corollary~2]{Schneiderman}.
We recall this characterization in Theorem~\ref{thm:Schneiderman}, but note that a knot~$K$ is stably slice if and only if $\op{Arf}(K)=0$~\cite{CochranOrrTeichner}.

\medbreak
This paper establishes lower bounds on $\operatorname{sn}(L)$ and, in the knot case, describes its relation to the $4$--genus.
Our first lower bound uses the multivariable signature and nullity~\cite{Cooper,CimasoniFlorens}. 
Consider the subset $\mathbb{T}^m=(S^1 \setminus \lbrace 1 \rbrace)^m \subset \C^m$ of the~$m$-dimensional torus.
Recall that for an $m$--component link $L$, the multivariable signature and nullity functions $\sigma_L,\eta_L \colon \mathbb{T}^m \to~\Z$ generalize the classical Levine-Tristram signature and nullity of a knot.
These invariants can either be defined using C-complexes or using $4$--dimensional interpretations~\cite{Cooper, Florens, Viro09, ConwayNagelToffoli, DFL18} and are known to be link concordance invariants on a certain subset $\mathbb{T}^m_! \subset \mathbb{T}^m$~\cite{NagelPowell,ConwayNagelToffoli}. 
Our first lower bound on the stabilizing number reads as follows.
\begin{customthm}{\ref{thm:S2xS2}}
\label{thm:S2xS2Intro}
If an $m$--component link $L$ is stably slice, then, for all $\omega \in \mathbb{T}^m_!$, we have $
|\sigma_L(\omega)|+|\eta_L(\omega)-m+1| \leq 2 \operatorname{sn}(L).$
\end{customthm}

As illustrated in Example~\ref{ex:6Compo}, Theorem~\ref{thm:S2xS2Intro} provides several examples in which the stabilizing number can be determined precisely.
Here, note that upper bounds on $\operatorname{sn}(L)$ can be often be computed using band pass moves; see Remark~\ref{rem:Kirby}.
The proof of Theorem~\ref{thm:S2xS2Intro} relies on the more technical statement provided by Theorem~\ref{thm:ImprovedGenusBound}. 
This latter result is a generalization of~\cite[Theorem 3.7]{ConwayNagelToffoli} which is itself a generalization of the Murasugi-Tristram inequality~\cite{Murasugi,Tristram, GilmerConfiguration,FlorensGilmer, Florens, CimasoniFlorens, Viro09, Powell}.
We state this result as it might be of independent interest, but refer to Definition~\ref{def:NullhomologousCobordism} for the definition of a nullhomologous cobordism.

\begin{customthm}{\ref{thm:ImprovedGenusBound}}
\label{thm:ImprovedGenusBoundIntro}
Let $V$ be a closed topological $4$--manifold with $H_1(V;\Z)=0$. If $\Sigma \subset (S^3 \times I) \csum V$ is a nullhomologous cobordism with $c$ double points between two $\mu$-colored links $L$ and~$L'$, then
\begin{equation*}
|\sigma_{L'}(\omega)-\sigma_{L}(\omega)+\op{sign}(V)| + |\eta_{L'}(\omega)-\eta_{L}(\omega)|-\chi(V)+2 
	\leq c-\sum_{i=1}^{\mu} \chi(\Sigma_i) 
\end{equation*}
for all $\omega\in \mathbb{T}_!^\mu$.
\end{customthm}
For $V=S^4$, Theorem~\ref{thm:ImprovedGenusBoundIntro} recovers the aforementioned (generalized) Murasugi-Tristram inequality, while the case $V=\C P^2$ is discussed in Example~\ref{ex:CP2}.

In the remainder of the introduction, we restrict to knots.
Analogously to $4$--genus computations, we are rapidly confronted to the following problem: if a knot~$K$ is algebraically slice, then its signature and nullity are trivial, and therefore the bound of Theorem~\ref{thm:S2xS2Intro} is ineffective. 
Pursuing the analogy with the $4$--genus, we obtain obstructions via the Casson-Gordon invariants.
We also briefly mention 
$L^{(2)}$-signatures in Remark~\ref{rem:HigherOrder}.

To state this obstruction, note that given a knot $K$ with $2$--fold branched cover $\Sigma_2(K)$ and a character $\chi$ on $H_1(\Sigma_2(K);\Z)$, Casson and Gordon introduced a signature invariant $\sigma(K,\chi)$ and a nullity invariant $\eta(K,\chi)$, both of which are rational numbers. We also use $\beta_K$ to denote the $\Q/\Z$--valued linking form on $H_1(\Sigma_2(K);\Z)$.
Our second obstruction for the stabilizing number is an adaptation of Gilmer's obstruction for the $4$-genus~\cite{GilmerGenus}.
\begin{customthm}{\ref{thm:CGstabilizing}}
\label{thm:CGstabilizingIntro}
If a knot $K$ bounds a locally flat disk $D$ in~$D^4 \csum n S^2 \times~S^2$, then the linking form $\beta_K$ can be written as a direct sum $\beta_1 \oplus \beta_2$ such that
\begin{enumerate}
\item $\beta_1$ has an even presentation matrix of rank $4n$ and signature $\sigma_K(-1)$;
\item there is a metabolizer $\beta_2$ such that for all characters of prime power order in this metabolizer, 
\[ |\sigma(K,\chi)+\sigma_K(-1)| \leq \eta(K,\chi)+4n+1. \]
\end{enumerate}
\end{customthm}

In Example~\ref{ex:AlgSlicesn2}, we use Theorem~\ref{thm:CGstabilizingIntro} to provide an example of an algebraically slice knot whose stabilizing number is at least two. 
Readers who are familiar with Gilmer's result might notice that Theorem~\ref{thm:CGstabilizingIntro} is weaker than the corresponding result for the $4$--genus: if $K$ bounds a genus $g$ surface in $D^4$, then Gilmer shows that $\beta_K=\beta_1 \oplus \beta_2$ where $\beta_1$ has presentation of rank $2g$, and not $4g$.

In view of Theorem~\ref{thm:S2xS2Intro} and Theorem~\ref{thm:CGstabilizingIntro}, one might wonder whether the stabilizing number is related to the $4$--genus. 
In fact, we show that the stabilizing number of an Arf invariant zero knot is always bounded above by the topological $4$-genus $g_4^{\op{top}}(K)$:
\begin{customthm}{\ref{thm:StablGenus}}
\label{thm:StablGenusIntro}
If $K$ is a knot with~$\op{Arf}(K) = 0$, then $\sn(K) \leq g_4^{\op{top}}(K)$.
\end{customthm}

The idea of the proof of Theorem~\ref{thm:StablGenusIntro} is as follows: start from a locally flat genus~$g$ surface $\Sigma \subset D^4$ with boundary $K$, use the Arf invariant zero condition to obtain a symplectic basis of curves of $\Sigma$, so that ambient surgery on half of these curves produces a disk $D$ in~$D^4 \csum g S^2 \times S^2$. 
To make this precise, one must carefully keep track of the framings and arrange that $D$ is nullhomologous.
\begin{remark} \label{rem:ArfSurfaceIntro}
During the proof of Theorem~\ref{thm:CGstabilizingIntro}, we construct a stable tangential framing of a surface~$\Sigma \subset D^4$ bounding a knot~$K$ such that $\operatorname{Arf}(K)$ can be computed from the bordism class $[\Sigma,f] \in \Omega_2^{\operatorname{fr}}\cong \Z_2$. 
	Although this result appears to be known to a larger or lesser degree~\cite{FreedmanKirby, Kirby89, Scorpan05}, we translated it from the context of characteristic surfaces as it might be of independent interest. 
\end{remark}

Motivated by the striking similarities between the bounds for $\sn(K)$ and $g_4^{\op{top}}(K)$, we provide an infinite family of knots with $4$--genus~$2$ but stabilizing number $1$; see Proposition~\ref{prop:Genus2Stab1} for precise conditions on the knots~$J_i$. 
\begin{proposition} \label{prop:Genus2stabilizingNumber1}
	If $J_1,J_2,J_3$ have large enough Levine-Tristram signature functions and vanishing nullity at all $15$--th roots of unity, then the knot $K:=R(J_1,J_2,J_3)$ described in Figure~\ref{fig:CExample} has $g_4^{\op{top}}(K)=2$ but $\sn(K)=1$.
\end{proposition}
\begin{figure}
\includegraphics{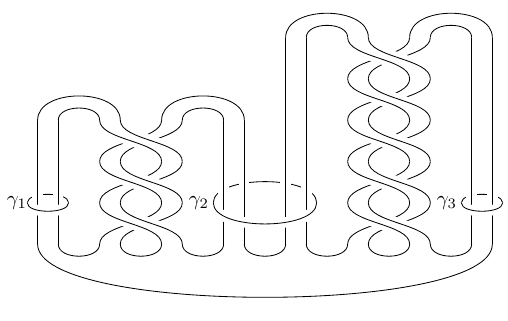}
\caption{The knot $R$ together with the curves~$\gamma_i$, which give rise to the infection $R(J_1, J_2, J_3)$.}
\label{fig:CExample}
\end{figure}

The idea of the proof of Proposition~\ref{prop:Genus2stabilizingNumber1} is as follows.
To show that $K$ has stabilizing number $1$, we use the following observation, which is stated in a greater generality in Lemma~\ref{lem:WindingPattern}:
if $K$ is obtained from a slice knot by a winding number zero satellite, then $K$ has stabilizing number~$1$. To show that $K$ has genus $2$, we use Casson-Gordon invariants: namely, we use the behavior of $\sigma(K,\chi)$ and $\eta(K,\chi)$ under satellite operations, as described in Theorem~\ref{thm:Abchir} and Proposition~\ref{prop:CGNullitySatellite}. Note that such a satellite formula for $\sigma(K,\chi)$ has appeared in~\cite{Abchir}, but unfortunately contains a mistake; see Example~\ref{ex:Motivation}.
\medbreak
We conclude this introduction with three questions.
\begin{question}\label{question:GenusLinks}
Does the inequality $\sn(L) \leq g_4^{\op{top}}(L)$ hold for stably slice links~$L$ of more than one component?
\end{question}
Question~\ref{question:GenusLinks} is settled in the knot case: Theorem~\ref{thm:StablGenusIntro} shows that $\sn(K) \leq~ g_4^{\op{top}}(K)$ holds for stably slice knots. We are currently unable to generalize this proof to links. 

\begin{question}\label{question:Smooth}
Does there exist a non-topological slice knot~$K$ such that the topological and smooth stabilizing numbers are distinct; i.e. $0 < \sn^\text{top}(K) <\sn^\text{smooth}(K)$?
\end{question}

\begin{question}\label{question:SpecificKnot}
Let $R(J,J)$ be the knot depicted in Figure~\ref{fig:KJ}.
Does the knot $K:=\#_{i=1}^3 R(J,J)$ have $\sn(K)=2$ or $\operatorname{sn}(K)=3$?
\end{question}
Using Theorem~\ref{thm:CGstabilizingIntro}, we can show that $\operatorname{sn}(K) \geq 2$. 
Using Casson-Gordon invariants, one can show that $g_4^{\op{top}}(K)=3$ and therefore Theorem~\ref{thm:StablGenusIntro} implies that $\operatorname{sn}(K) \in \lbrace 2,3 \rbrace$. 
We are currently not able to decide on the value of $\operatorname{sn}(K)$.
\medbreak
This article is organized as follows. 
In Section~\ref{sec:StablySliceLinks}, we define the stabilizing number~$\sn(L)$ and investigate its relation to band pass moves and winding number zero satellite operations.
In Section~\ref{sec:Proof}, we review the multivariable signature and nullity and prove Theorem~\ref{thm:S2xS2Intro}.
In Section~\ref{sec:CG}, we recall the definitions of the Casson-Gordon invariants and prove Theorem~\ref{thm:CGstabilizingIntro}.
In Section~\ref{sec:4Genus}, we discuss various definitions of the Arf invariant and prove Theorem~\ref{thm:StablGenusIntro}.
Finally, Appendix~\ref{Appendix} provides satellite formulas for the Casson-Gordon signature and nullity invariants.

\subsection*{Acknowledgments}
We thank Mark Powell for providing the impetus for this project and for several extremely enlightening conversations. 
We are also indebted to Peter Feller for asking us about the relationship between $g_4^{\op{top}}(K)$ and $\operatorname{sn}(K)$.
We also thank an anonymous referee for many helpful comments.
AC thanks Durham University for its hospitality and was supported by an early Postdoc.Mobility fellowship funded by the Swiss FNS. MN thanks the Université de Genève for its hospitality. MN was supported by the ERC grant 674978 of the European Union's Horizon 2020 program.

\section{Stably slice links} 
\label{sec:StablySliceLinks}
In this subsection, we define the notion of a stably slice link. After discussing the definition, we give some examples and recall a result of Schneiderman which gives a necessary and sufficient condition for a link to be stably slice.
\medbreak

Let $W$ be a $4$--manifold with boundary~$\partial W \cong S^3$. We say that a properly embedded disk~$\Delta$ is \emph{nullhomologous}, if its fundamental class~$[\Delta, \partial \Delta] \in H_2(W, \partial W; \Z)$ vanishes. By Poincaré duality, $\Delta$ is nullhomologous if and only if $\Delta \cdot \alpha=0$ for all~$\alpha \in H_2(W; \Z)$, where $\cdot$ denotes algebraic intersections.

The next definition introduces the main notions of this article.

\begin{definition}\label{def:StablySlice}
A link $L \subset S^3$ is \emph{stably slice} if there exists $n\geq 0$ such that the components of~$L$ bound a collection of disjoint locally flat nullhomologous disks in the manifold~$D^4 \csum n S^2 \times S^2$. The \emph{stabilizing number} $\sn(L)$ of a stably slice link is the minimal such $n$.
\end{definition}

\begin{remark}\label{rem:EquivalentDef}
To put Definition~\ref{def:StablySlice} in the setting of the article~\cite{Schneiderman}, observe that it can can be rephrased as follows:
a knot is stably slice if and only if it bounds a smoothly \emph{immersed} 
disk~$D$ in~$D^4$ that can be homotoped to a locally flat embedding in~$D^4 \csum n S^2 \times S^2$. 
In one direction, if $K$ bounds such a disk in $D^4$, then the resulting embedded disk in~$D^4 \csum n S^2 \times S^2$ is automatically nullhomologous, and therefore~$K$ is stably slice. Conversely, we assume~$K$ is stably slice, and produce the required disk~$D \subset D^4.$
Since the pair~$W = D^4 \csum n S^2 \times S^2$ and $\partial W = S^3$ is $1$--connected and $\pi_1(\partial W) =~0$, deduce that $\pi_2(W,\partial W) \to H_2(W, \partial W;\Z)$ is an isomorphism by the relative Hurewicz theorem. 
Consequently, any nullhomologous disk~$\Delta$ can be homotoped into~$D^4$,
while fixing its boundary~$K$.
 By the Whitney immersion theorem, arrange the resulting map~$D^2 \rightarrow D^4$ to be a smooth proper immersion; see e.g.~\cite[Section~7.1]{Ranicki02}.
\end{remark}

Using Norman's trick, any knot $K \subset S^3$ bounds a locally flat embedded disk in~$D^4 \csum S^2\times S^2$~\cite[Theorem 1]{Norman69}; this explains why we restrict our attention to \emph{nullhomologous} disks.

The next lemma reviews Norman's construction in the case of links.
\begin{lemma}[Norman's trick]\label{lem:Norman}
Any $m$--component link $L \subset S^3$ bounds a disjoint union of locally flat disks in $D^4 \csum mS^2 \times S^2$.

Furthermore, for every $k \in \Z$ and in the case of a knot~$L = K$, such a disk can be arranged to represent the class $[ \lbrace \pt \rbrace \times S^2 ] + k [ S^2 \times \lbrace \pt \rbrace ]$.
\end{lemma}
\begin{proof}
Pick locally flat immersed disks~$\Delta = \Delta_1 \cup \cdots \cup \Delta_m$ in~$D^4$ with boundary the link~$L = K_1 \cup \cdots \cup~K_m$, and which only intersect 
in transverse double points.
For each disk~$\Delta_i$, connect sum~$D^4$ with a separated copy of~$S^2 \times S^2$. 
Now connect sum each disk~$\Delta_i$ into the~$\lbrace \pt \rbrace \times S^2$ of its corresponding~$S^2 \times S^2$ summand. 
This way, each disk~$\Delta_i$ has a dual sphere~$S^2 \times \lbrace \pt \rbrace$, that is a sphere that intersects~$\Delta_i$ geometrically in exactly one point and no other disks.
As a consequence, a meridian of~$\Delta_i$ bounds the punctured~$S^2 \times \lbrace \pt \rbrace$, which is disjoint from~$\Delta$. 
The dual sphere~$S^2 \times \lbrace \pt \rbrace$ has a trivial normal bundle, so we can use push-offs of it. 
This implies that a finite collection of meridians of~$\Delta$ will bound disjoint disks in the complement of $\Delta$. 

At each double point~$x$ of $\Delta$ pick one of the two sheets. 
This sheet will intersect a tubular neighborhood of the other sheet in a disk~$D_x$, whose boundary is the meridian~$\mu_x$.
For the collection of meridians~$\{\mu_x\}$, find disjoint disks~$D'_x$ as above, namely by tubing into a parallel push-off of $S^2 \times \lbrace \pt \rbrace$.
Replace~$D_x$ with $D_x'$ to remove all immersion points, and call the resulting disjointly embedded disks~$\Delta'_i$.

For the second part of the lemma, pick the immersed disk~$\Delta$ for $K$ to have self-intersection points whose signs add up to $k$ by adding trivial local cusps; see e.g.~\cite[Figure 2.4]{Scorpan05}.
Tubing $\Delta$ into the dual sphere as above results in a disk~$\Delta'$, which represents the class
$[\lbrace \pt \rbrace \times S^2] + k[S^2 \times \lbrace \pt \rbrace]$.
\end{proof}

\begin{figure}[!htb]
\begin{tikzpicture}
[   x={(0:2cm)},
    y={(90:0.5cm)},
    z={(-100:1cm)}
]
\tikzset{->-/.style={decoration={markings,mark=at position 0.2 with
    {\arrow{>}}},postaction={decorate}}}
\tikzset{-<-/.style={decoration={markings,mark=at position 0.2 with
    {\arrow{<}}},postaction={decorate}}}
\begin{scope}[shift={(-1.5,0)}]
\node[left] at (-1,0.5,-1) {$K$};
\draw[->-] (1,0,1) -- (-1,0,-1);
\draw[-<-] (1,1,1) -- (-1,1,-1);

\draw[line width=5pt, white] (-1,0,1) -- (1,0,-1);
\draw[line width=5pt, white] (-1,1,1) -- (1,1,-1);
\node[left] at (-1,0.5,1) {$K'$};
\draw[->-] (-1,0,1) -- (1,0,-1);
\draw[-<-] (-1,1,1) -- (1,1,-1);
\end{scope}
\begin{scope}[shift={(1.5,0)}]
\draw[->-] (-1,0,1) -- (1,0,-1);
\draw[-<-] (-1,1,1) -- (1,1,-1);
\draw[line width=5pt, white] (1,0,1) -- (-1,0,-1);
\draw[line width=5pt, white] (1,1,1) -- (-1,1,-1);
\draw[->-] (1,0,1) -- (-1,0,-1);
\draw[-<-] (1,1,1) -- (-1,1,-1);
\end{scope}
\end{tikzpicture}
\caption{A band pass move. Two strands that bound a band belong to the same link components~$K$ and $K'$.}\label{fig:BandPass}
\end{figure}
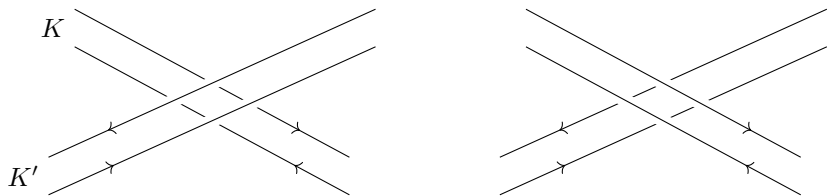

Our next goal is to provide examples of stably slice links using band pass moves. Recall that a \emph{band pass} is the local move on a link diagram depicted in Figure~\ref{fig:BandPass}. The associated equivalence relation on links has been studied by Martin, and it agrees with $0$--solve equivalence~\cite[Section~4]{Martin15}. 

We wish to show that if a link is band pass equivalent to 
a strongly slice link, then it is stably slice. To achieve this, we introduce a relative version of stable sliceness.
\begin{definition}\label{def:StableConcordance}
Let $L = K_1 \cup \cdots \cup K_m$ and $L' = K_1'\cup \cdots\cup K_m' \subset S^3$ be two links. An \emph{$n$--stable concordance} is a collection of disjointly locally flat annuli~$\Sigma_1, \ldots, \Sigma_m \subset S^3 \times I \csum n S^2 \times S^2 =: W$ such that
\begin{enumerate}
\item the knots $K_i$ and $K_i'$ cobound the surface $\Sigma_i$, that is
\[ \partial \Sigma_i = -K_i \sqcup K_i' \subset -(S^3 \times\{-1\}) \sqcup (S^3 \times \{1\});\]
\item the fundamental class~$[\Sigma_i,\partial \Sigma_i] \in H_2(W, \partial W;\Z)$ vanishes.
\end{enumerate}
\end{definition} 

In a nutshell, a stable concordance is a concordance in~$S^3 \times I \csum n S^2 \times S^2$ that intersects zero algebraically the $2$-spheres of $n S^2 \times S^2$. 
The next lemma establishes a relation between stable concordance and stable sliceness.
\begin{lemma}\label{lem:StablyConcordantUnlinkStablySlice}
Let $\Sigma$ be an $n$--stable concordance between links~$L$ and~$L'$. Let $\Delta$ be a collection of nullhomologous slice disks
in~$D^4 \csum n' S^2 \times S^2$ for~$L'$. Then $\Sigma \cup \Delta$ forms a collection of nullhomologous slice disk for~$L$ in 
\[ \big( S^3 \times I \csum nS^2 \times S^2 \big) \cup_{S^3 \times \lbrace 1 \rbrace} \big( D^4 \csum n' S^2 \times S^2 \big) = D^4 \csum (n+n') S^2 \times S^2.\]
In particular, if $L$ is stably concordant to 
a strongly slice link,
then $L$ is stably slice.
\end{lemma}
\begin{proof}
Since $\Sigma \cup \Delta$ consists of disjoint disks, we need only verify that each disk~$\Sigma_i \cup~\Delta_i$ is nullhomologous.
Set $Z:=S^3 \times I \csum n S^2 \times S^2$ and $X :=D^4 \csum n' S^2 \times S^2$, and define $W:= Z \cup_{S^3 \times \lbrace 1 \rbrace} X$.
Consider the following portion of the long exact sequence 
of the triple~$\partial W \subset \partial Z \subset~W$:
\begin{equation}\label{eq:NullhomologousAdditiv}
	\begin{tikzcd}[column sep = 0.5cm]
            \to H_2(\partial Z, \partial W; \Z) \ar[r] & H_2(W, \partial W; \Z) \ar[r] \ar[dr, dashed] & H_2(W, \partial Z; \Z) \ar[d,"\op{exc}"] \to\\
                                                       &&H_2(Z, \partial Z;\Z) \oplus H_2(X, \partial X;\Z).
\end{tikzcd} 
\end{equation}
The vertical arrow is an excision isomorphism: we thicken~$S^3 \times \{1\}$ and remove its interior. 
This splits~$W$ into the disjoint union of~$Z$ and~$X$.
Consequently, the dashed arrow sends the class~$[\Sigma_i \cup \Delta_i,\partial (\Sigma_i \cup \Delta_i)]$ to~$[\Sigma_i,\partial \Sigma_i] + [\Delta_i,\partial \Delta_i]$.
By assumption, both of these summands are zero.

The dashed arrow is injective, since~$H_2(\partial Z, \partial W; \Z) \cong H_2(S^3; \Z) = 0$.
Deduce that~$[\Sigma_i \cup \Delta_i,\partial (\Sigma_i \cup \Delta_i)] = 0$, and thus each disk~$\Sigma_i \cup \Delta_i$ is nullhomologous, as~desired.
\end{proof}

Next, we show how band pass moves give rise to stable concordances.
\begin{remark}
\label{rem:Kirby}
If an oriented link $L'$ is obtained from an oriented link $L$ by a single band pass move, then~$L'$ and $L$ are $1$-stable concordant.
Indeed, the annuli are obtained as the trace of the isotopy near the $2$--handles of $S^3 \times I \csum S^2 \times S^2$ depicted in Figure~\ref{fig:BandConcordanceKirby};
these annuli are nullhomologous provided we move a pair of strands with \emph{opposite} orientations over the $2$--handles.
\begin{figure}[!htb]
\includegraphics{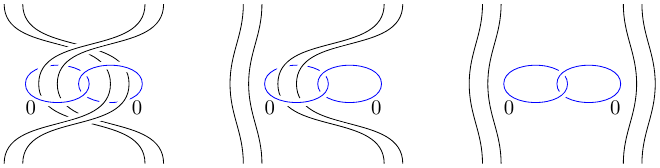}
\caption{A cylinder in $S^3 \times I \csum  S^2 \times S^2$ depicted in various
cross sections.}\label{fig:BandConcordanceKirby}
\end{figure}
\end{remark}

Next, we use band pass moves to provide examples of stably slice links.
\begin{example}\label{ex:BingHopf}
\begin{figure}[!htb]
\includegraphics{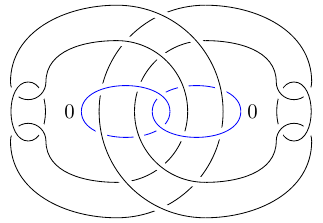}
\caption{Bing double of the Hopf link with $0$--handles.}
\label{fig:BingHopf}
\end{figure}
Consider the $4$--component link~$L$ depicted in Figure~\ref{fig:BingHopf}, which is the Bing double of the Hopf link. We claim that the stabilizing number of $L$ is $1$. 
Sliding the link over the depicted $2$--handles (or doing the associated band pass), we deduce that~$L$ is slice in $D^4 \csum S^2 \times S^2$. 
The Milnor invariant~$|\mu_{1234}(L)|=1$ shows that $L$ is not slice in $D^4$; here we used the behavior of the Milnor invariants under Bing doubling~\cite[Theorem 8.1]{CochranDerivatives} and the fact that the Milnor invariants are concordance invariants~\cite{Casson}.
Therefore $\sn(L)=1$, as claimed.
\end{example}

In Example~\ref{ex:BingHopf}, we converted crossing changes into band passes via Bing doubling.
In order to generalize from Bing doubles to arbitrary winding number~$0$ satellite operations,
we first recall the set-up for satellites.
Let $R \subset S^3$ be a knot together with an \emph{infection curve}~$\gamma \subset S^3 \sm R$, which is an unknot~$\gamma \subset S^3$ and an integer~$k \in \Z$.
We identify $\nu (\gamma) = S^1 \times D^2$ via the Seifert framing on $\gamma$ so that the exterior~$S^3 \sm \nu (\gamma)$ has an evident product structure~$S^1\times D^2$, and we may consider~$R$ as a knot in~$S^1 \times D^2$.
Now suppose we are given another knot~$J \subset S^3$. 
Fix the trivialization~$\nu (J) \cong J \times D^2$ that corresponds to the integer~$k$ in the Seifert framing, i.e. the framing which maps $J \times \{1\}$ to the curve~$k \mu_J + \lambda_J$, where $\lambda_J$ is the Seifert longitude.
Keeping this identification in mind, the \emph{satellite}~$R(J;\gamma, k)$ is the knot 
\[ R \subset S^1 \times D^2 = \nu (J) \subset S^3.\] 
The exterior of $R(J;\gamma,k)$
is $X_{R(J; \gamma, k)} = X_R \setminus \nu(\gamma) \cup X_J$, where we identify the boundary tori via
$\mu_{\gamma} \mapsto k\mu_J + \lambda_J$ and $\lambda_{\gamma} \mapsto \mu_J$.
We say $R \subset S^1 \times D^2$ is the \emph{pattern}, and the knot~$J$ is the \emph{companion}.  
The linking number~$\lk(R, \gamma)$, which coincides with the (algebraic) intersection number~$R \cdot (\{\pt\} \times D^2)$, is called the \emph{winding number}. 
We also write $R(J; \gamma)$ instead of $R(J; \gamma,0)$.
\begin{figure}
	\includegraphics{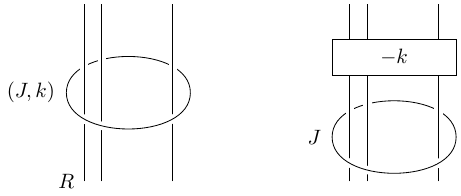}
	\caption{Twice the knot $R(J; \gamma, k)$; on the right as an infection $R'(J; \gamma)$, where $R'$ is obtained from $R$ by adding $k$ full left-handed twists. }
	\label{fig:infection}
\end{figure}

The next lemma describes the effect of winding number $0$ satellite operations on stable concordance. 
\begin{lemma}\label{lem:WindingPattern}
Let $R$ be a winding number~$0$ pattern, $J$ a knot, and $k \in \Z$.
Then the satellite~$R(J; \gamma, 2k)$ is $1$--stable concordant to~$R = R(U; \gamma)$.
\end{lemma}
\begin{proof}
Use the last assertion of Lemma~\ref{lem:Norman} to pick a locally flat disk $D$ in $D^4 \csum S^2 \times~S^2$ bounding the companion knot~$J$, and which represents the homology class~$x:=k [S^2 \times \lbrace \pt \rbrace] + [\lbrace \pt \rbrace \times S^2]$. 
While we require that $D$ is embedded, it need not be nullhomologous.
Note that~$J$ is the core~$S^1 \times \{0\}$ of the solid torus~$S^1 \times D^2 =~\nu (J)$.  
Next, write~$i_R \colon R \rightarrow S^1 \times D^2 =~\nu (J)$ for the inclusion of the pattern knot~$R$. 
Remove a small $4$-ball from the interior of $D^4$ in such a way that the disk~$D \subset D^4 \csum S^2 \times S^2$ becomes an annulus~$A \subset (S^3 \times I) \csum S^2 \times S^2 =: W$ with boundary the disjoint union of~$J$ and the unknot~$U$.
Under the isomorphism~$H_2(D^4 \csum S^2 \times S^2, S^3; \Z) \cong H_2(S^2 \times S^2; \Z)$, the class~$x$ has intersection product~$x \cdot x=2k$.
Since $D$ is embedded and~$x \cdot x =2k$, the (relative) Euler number of $D$ is $2k$.
Thus, the unique trivialization~$\nu (D)\cong D\times D^2$ induces the $2k$--framing on the knot~$J$.
Consequently, the annulus~$A$ admits a framing~$n_A \colon A \times D^2 \xrightarrow{\sim} \nu A$ that restricts to the $2k$--framing on~$J$ and to the $0$--framing on~$U$.

Now consider the annulus~$\Sigma$ given by 
\[ \Sigma = R \times I \xrightarrow{i_R \times \id_I} (S^1 \times D^2) \times I \xrightarrow{\text{flip}} (S^1 \times I) \times D^2 = A \times D^2 \xrightarrow{n_A} \nu A \subset W.\] 
Observe that $R \times \{-1\} = R(J; \gamma, 2k) \subset S^3 \times \{-1\}$ and $R \times \{1\} = R(U; \gamma) \subset S^3 \times \{1\}$. 
It only remains to check that $\Sigma$ is nullhomologous, i.e. that $[R \times I,\partial (R \times I)]$ vanishes in $H_2(W, \partial W ; \Z)$. This follows from the fact that the class 
\[ [R \times I,\partial (R \times I)] \in H_2\big( S^1 \times D^2 \times I, S^1 \times D^2 \times \{\pm 1\}; \Z\big)\] 
already vanishes, since~$R$ has winding number~$0$.
\end{proof}

Having provided some examples and constructions of stably slice links, we conclude this subsection by recalling Schneiderman's characterization of stably slice links~\cite{Schneiderman}; see also~\cite[Theorem 1.1]{Martin15}.
\begin{theorem}[Schneiderman]\label{thm:Schneiderman}
For a link $L \subset S^3$, the following are equivalent:
\begin{enumerate}
\item the link~$L$ is stably slice;
\item the following invariants vanish: the pairwise linking numbers of $L$, the triple linking numbers $\mu_{ijk}(L)$, the mod~$2$ Sato-Levine invariants of~$L$, and the Arf invariants of the components of~$L$.
\end{enumerate}
\end{theorem}
\begin{proof}
The pairwise linking numbers of $L$ can be computed via the number of intersection points among nullhomologous surfaces bounding $L$.
In particular, stably slice links have pairwise vanishing linking numbers, and whenever this latter condition is satisfied, Schneiderman's \emph{tree-valued invariant} $\tau_1(L)$ is defined.
In his terminology~\cite[Section 1.2]{Schneiderman}, an $m$--component link $L \subset S^3$ is stably slice if it bounds a collection of properly immersed disks $D_1,\ldots,D_m$ such that for some~$n$, the~$D_i$ are homotopic (rel boundary) to pairwise disjoint embeddings in the connected sum $D^4 \csum nS^2 \times S^2$.
As explained in Remark~\ref{rem:EquivalentDef}, this is equivalent to Definition~\ref{def:StablySlice}.

Using this definition of stable sliceness, \cite[Corollary 2]{Schneiderman} shows that $L$ is stably slice if and only if 
$\tau_1(L)$ vanishes. Using~\cite[Theorem 1]{Schneiderman} and~\cite[Theorem 1.1]{CST}, this is equivalent to the triple Milnor linking numbers of $L$, the mod~$2$ Sato-Levine invariants of $L$, and the Arf invariants of the components of~$L$ all vanishing.
\end{proof}

We conclude  by mentioning two additional characterizations of stable sliceness.
\begin{remark}
Combining Theorem~\ref{thm:Schneiderman} with a result of Martin~\cite[Corollary~1]{Martin15}, a link is stably slice if and only if it is band pass equivalent to the unlink, if and only if it is $0$-solvable, a notion due to Cochran-Orr-Teichner~\cite{CochranOrrTeichner}.
\end{remark}

\section{Abelian invariants}\label{sec:Proof}
In Theorem~\ref{thm:S2xS2}, we establish a bound on the stabilizing number in terms of the multivariable signature and nullity.
The main technical ingredient is Theorem~\ref{thm:ImprovedGenusBound}, which provides genus restrictions for nullhomologous cobordisms.
This section is organized as follows. 
In Section~\ref{sub:BackgroundAbelian}, we briefly review the multivariable signature and nullity,
and in Section~\ref{sub:S2S2Abeliean}, we prove Theorems~\ref{thm:ImprovedGenusBound} and~\ref{thm:S2xS2} .

\subsection{Background on abelian invariants}\label{sub:BackgroundAbelian}

We briefly recall a $4$--dimensional interpretation of the multivariable signature and nullity. We refer to~\cite{CimasoniFlorens} for the definition in terms of C-complexes.
\medbreak
We start with some generalities on twisted homology. Let $(X,Y)$ be a CW-pair, let $\varphi \colon \pi_1(X) \to \mathbb{Z}^\mu=\langle t_1,\dots,t_\mu \rangle$ be a homomorphism, and let $\omega = (\omega_1,\dots,\omega_\mu)$ be an element of $\mathbb{T}^\mu:=\big(S^1 \sm \{1\}\big)^\mu \subset~\C^\mu$. 
Compose the induced map
$\mathbb{Z}[\pi_1(X)] \to~\mathbb{Z}[\mathbb{Z}^\mu]$ with the
map~$\mathbb{Z}[\mathbb{Z}^\mu] \xrightarrow{\alpha}~\C$ which evaluates $t_i$ at $\omega_i$ to get a morphism~$\phi \colon \mathbb{Z}[\pi_1(X)] \to \mathbb{C}$ of rings with involutions.
In turn, $\phi$ endows $\C$ with a $(\mathbb{C},\mathbb{Z}[\pi_1(X)])$--bimodule structure.
To emphasize the choice of $\omega$, we shall write $\mathbb{C}^\omega$ for this bimodule. 
We denote the universal cover of $X$ as $p \colon \widetilde{X} \to X$, and set~$\widetilde{Y}:=p^{-1}(Y)$, so that  $C\big(\widetilde{X},\widetilde{Y} \big)$ is a left $\mathbb{Z}[\pi_1(X)]$--module.
Since $\mathbb{C}^\omega$ is a $(\mathbb{C},\mathbb{Z}[\pi_1(W)])$--bimodule, we may consider the homology groups 
\[ H_k(X,Y;\mathbb{C}^\omega)=H_k\big(\mathbb{C}^\omega \otimes_{\Z[\pi_1(X)]} C\big(\widetilde{X},\widetilde{Y}\big)\big),\]
which are complex vector spaces. We now describe two examples that we will use constantly in the remainder of this subsection.

\begin{example} \label{ex:LinkExterior}
A \emph{$\mu$--colored link} is an oriented link $L$ in $S^3$ whose components are partitioned into $\mu$ sublinks $L_1 \cup \cdots \cup L_\mu$.
Consider the exterior~$X_L$ of a $\mu$--colored link~$L$ in $S^3$ together with the morphism~$\pi_1(X_L) \to  \Z^\mu, \gamma \mapsto (\ell k (\gamma,L_1),\ldots,\ell k (\gamma,L_\mu))$.
 Given~$\omega~\in~\mathbb{T}^\mu$, we can construct the complex vector spaces~$H_i(X_L;\C^\omega)$ as described above.
\end{example}

\begin{example} \label{ex:ColoredBoundingSurface}
Let $L \subset S^3$ be a $\mu$--colored link, and let $F = F_1 \cup \cdots \cup F_\mu \subset~D^4$ be a collection of connected locally flat surfaces that intersect transversely in double points, and~$\partial F_i = L_i$. 
	We refer to $F$ as a \emph{colored bounding surface} for $L$.
In this article, we assume that the $F_i$ are connected, but refer to~\cite{ConwayNagelToffoli} for the general case.
The exterior~$D_F$ of $F$ is a $4$--manifold, whose~$H_1(D_F; \Z)\cong \Z^\mu$ is freely generated by the meridians $m_1,\ldots,m_\mu$ of~$F$; see e.g. Lemma~\ref{lem:FreelyGenerated} below.
Mapping the meridian $m_i$ to the $i$--th canonical basis vector of $\Z^\mu$ defines a homomorphism~$H_1(D_F; \Z) \to \Z^\mu$.
The complex vector space~$H_2(D_F;\C^\omega)$ is equipped with a $\C$--valued twisted intersection form, whose signature we denote by $\text{sign}_\omega(D_F)$; we refer to~\cite[Section 2 and 3]{ConwayNagelToffoli} for further details.
\end{example}

Next, we recall a definition of the multivariable signature and nullity, which uses the coefficient systems of Examples~\ref{ex:LinkExterior} and~\ref{ex:ColoredBoundingSurface}.

\begin{definition}\label{def:MultivariableSignatureNullity}
Let $L$ be a $\mu$--colored link and let $\omega \in \mathbb{T}^\mu$.
The \emph{multivariable nullity} of $L$ at $\omega$ is defined as $\eta_L(\omega):=\dim_\C H_1(X_L;\C^\omega)$.
Given a colored bounding surface~$F$ for $L$ as in Example~\ref{ex:ColoredBoundingSurface}, the \emph{multivariable signature} of~$L$ at $\omega$ is defined as the signature $\sigma_L(\omega) := \sign_\omega(D_F)$.
\end{definition}

It is known that $\sigma_L(\omega)$ does not depend on the choice of the colored bounding surface $F$, and that it coincides with the original definition given by Cimasoni-Florens~\cite{CimasoniFlorens}; see~\cite[Proposition 3.5]{ConwayNagelToffoli}.

In fact, Degtyarev, Florens and Lecuona showed that~$\sigma_L(\omega)$ can be computed using other ambient spaces than $D^4$~\cite{DFL18}. 
We recall this result.
In a topological $4$--manifold~$W$, we still refer to a collection of locally flat and connected surfaces~$F = F_1 \cup \cdots \cup F_\mu \subset W$ as a colored bounding surface if the surfaces intersect transversally and at most in double points that lie in the interior of~$W$.
We denote the exterior of a colored bounding surface $F$ in~$W$ by~$W_F~=~W \setminus \bigcup_i \nu(F_i)$.  

The next lemma is used to define a $\C^\omega$-coefficient system on $W_F$.

\begin{lemma}\label{lem:FreelyGenerated}
Let~$W$ be a topological $4$--manifold with~$H_1(W; \Z) = 0$.
If~$F = F_1 \cup \cdots \cup F_\mu$ is 
a colored bounding surface such that $[F_i, \partial F_i] \in H_2(W, \partial W; \Z)$ is zero for each~$i$, then the homology~$H_1(W_F; \Z)$ of the exterior~$W_F$ is freely generated by the meridians of the $F_i$.
\end{lemma}
\begin{proof}
See~\cite[Lemma~4.1]{DFL18}.
\end{proof}

Using Lemma~\ref{lem:FreelyGenerated} and its notations, we construct the twisted homology groups $H_*(W_F;\C^\omega)$ and the twisted signature $\sign_\omega(W_F)$ just as in Example~\ref{ex:ColoredBoundingSurface}. The result of Degtyarev, Florens and Lecuona now reads as follows.
\begin{theorem}\label{thm:OtherAmbientSpaces}
Let~$W$ be a topological $4$--manifold with~$\partial W = S^3$ and~$H_1(W;\Z) =~0$.
Let~$F = F_1 \cup \cdots \cup F_\mu$ be a colored bounding surface for a colored link $L = L_1 \cup \cdots \cup L_\mu$.
If $[F_i,\partial F_i] \in H_2(W, \partial W; \Z)$ is zero for
each~$i$, then
\[ \sigma_L(\omega) = \sign_\omega(W_F) - \sign(W_F) = \sign_\omega(W_F) - \sign(W),\]
where $W_F$ denotes the exterior of the collection~$F$.
\end{theorem}
\begin{proof} See~\cite[Theorem~4.7 and (4.6)]{DFL18}. A close inspection of these arguments shows that they carry through to the topological category.
\end{proof}

The $4$--dimensional interpretations of both the Levine-Tristram and multivariable signature were originally stated using finite branched covers and certain eigenspaces associated to their second homology group~\cite{ViroOld, CimasoniFlorens}; the use of twisted homology only emerged later~\cite{Viro09,Powell, ConwayNagelToffoli,DFL18}. 
We briefly recall the construction of the aforementioned eigenspaces in the one variable case.
The multivariable case is discussed in~\cite{CimasoniConway}.
Let $X$ be a CW-complex and let $\varphi \colon H_1(X) \to  \Z_n=:G$ be an epimorphism, where we write $\Z_n:=\Z/n\Z$.
The homology groups of the $G$--cover $X_G$ induced by $\varphi$ are endowed with the structure of a $\C[G]$--module, and $H_*(X_G;\C)=H_*(X;\C[G])$.
Use~$t$ to denote a generator of $G$, and let $\omega$ be a root of unity.
 Consider the $\C$--vector space
\[ \op{Eig}(H_k(X;\C[G]), \omega) =\lbrace x \in  H_k(X;\C[G]) \ | \ tx=\omega x \rbrace. \]
This space will be referred to as the \emph{$\omega$--eigenspace} of $H_k(X;\C[G])$. 
Observe $\C$ is both a $\Z[\pi_1(X)]$--module and  a $\Z[G]$--module; in both cases we write $\C^\omega$.
The following lemma will be useful in Subsection~\ref{sub:CGstabilizing} below.

\begin{lemma}\label{lem:Eigenspaces}
The $\C$-vector spaces~$\op{Eig}(H_k(X;\C[G]), \omega)$ and $H_k(X;\C^\omega)$ are canonically isomorphic.
\end{lemma}
\begin{proof}
The subspace~$\op{Eig}(H_k(X;\C[G]),\omega)$ is isomorphic to $\C^\omega \otimes_{\C[G]} H_k(X;\C[G])$; see e.g.~\cite[Proposition 3.3]{CimasoniConway}.  
Since~$G$ is finite, Maschke's theorem implies that~$\C[G]$ is a semisimple ring~\cite[Chapter 1, Section 5.7]{Wisbauer}. 
Consequently, all its (left)-modules are projective and in particular flat~\cite[Theorem 4.2.2]{Weibel}. Combining these two observations, it follows that 
\[ \op{Eig}(H_k(X;\C[G]),\omega) \cong \C^\omega \otimes_{\C[G]} H_k(X;\C[G]) \cong H_k(\C^\omega \otimes_{\mathbb{C}[G]} C(X_G)).\]
Using this isomorphism and the associativity of the tensor product, we obtain the announced equality: 
\begin{align*}
\op{Eig}(H_k(X_G;\C),\omega)
&=H_k(\C^\omega \otimes_{\mathbb{C}[G]} C(X_G))\\
&= H_k\left( \mathbb{C}^\omega \otimes_{\mathbb{C}[G]} \mathbb{C}[G] \otimes_{\mathbb{Z}[\pi_1(X)]} C\big(\widetilde{X}\big) \right) \\
&=H_k\left( \mathbb{C}^\omega \otimes_{\mathbb{Z}[\pi_1(X)]} C\big(\widetilde{X}\big) \right)= H_k(X;\mathbb{C}^\omega). \qedhere
\end{align*}
\end{proof}

\subsection{Nullhomologous cobordism and the lower bound.}\label{sub:S2S2Abeliean}
The aim of this subsection is to use the multivariable signature and nullity in order to obtain a lower bound on the stabilizing number of a link.
\medbreak

Let $V$ be a closed topological $4$--manifold with~$H_1(V;\Z)=0$ and set~$W=V \csum (S^3 \times I)$. Observe that the boundary~$\partial W = S^3 \times \{1\} \sqcup -\big( S^3 \times \{-1\} \big)$ consists of two copies of~$S^3$. 
\begin{definition}\label{def:NullhomologousCobordism}
A \emph{nullhomologous cobordism} from a $\mu$--colored link~$L$
to a~$\mu$--colored link~$L'$ is a collection of locally flat surfaces~$\Sigma = \Sigma_1 \cup \cdots \cup \Sigma_\mu$ in~$W = V \csum (S^3 \times I)$ that have the following properties:
\begin{enumerate}
\item each surface~$\Sigma_i$ is connected and has boundary
\[ L'_i \times\{1\} \sqcup -L_i \times \{-1\} \subset S^3 \times \{1\} \sqcup -\big(S^3 \times \{-1\}\big),\]
\item each surface $\Sigma_i$ is embedded and the surfaces intersect transversally and at most in double points that lie in the interior of~$W$,
\item\label{item:nullhom} each class~$[\Sigma_i, \partial \Sigma_i] \in  H_2(W, S^3 \times I \sm \nu (D^4); \Z)$ is zero, where $D^4 \subset \operatorname{Int}(S^3 \times I)$ is the 4-ball leading to the connected sum~$V \csum (S^3 \times I)$.
\end{enumerate}
\end{definition}
In contrast to~\cite{ConwayNagelToffoli}, we require the surfaces~$\Sigma_i$ to be connected. This simplifies notation, and the extra generality is unnecessary for the later applications. 
We list three natural choices for~$V$. 
The $4$--sphere~$V=S^4$, where~$W$ is~$S^3 \times I$ and condition~(\ref{item:nullhom}) is automatic. 
Another choice is~$V = S^2 \times S^2$.
This time, colored cobordisms are not automatically nullhomologous; but we impose condition~(\ref{item:nullhom}). 
The same holds for the cases $V = {\C P}^2$ and $\ol{\C P}^2$ that we study in Example~\ref{ex:CP2} below.

Let $U \subset \Z[t_1^{\pm 1}, \ldots, t_\mu^{\pm 1}]$ be the subset of all polynomials with $p(1,\ldots,1) = \pm 1$, and define
\[ \mathbb{T}^\mu_! = \{ \omega \in \mathbb{T}^\mu \colon p(\omega) \neq 0 \text{ for } p \in U\}.\]
This set is a multivariable generalization of the concordance roots studied in~\cite{NagelPowell}; we refer to~\cite[Section 2.4]{ConwayNagelToffoli} for a more thorough discussion.

The next theorem provides an obstruction for two links to cobound a nullhomologous cobordism.
\begin{theorem}\label{thm:ImprovedGenusBound}
Let $V$ be a closed topological $4$--manifold with $H_1(V;\Z)=0$. If $\Sigma \subset (S^3 \times I) \csum V$ is a nullhomologous cobordism with $c$ double points between two $\mu$-colored links $L$ and~$L'$, then
\begin{equation}
\label{eq:ImprovedGenusBound}
|\sigma_{L'}(\omega)-\sigma_{L}(\omega)+\op{sign}(V)| + |\eta_{L'}(\omega)-\eta_{L}(\omega)|-\chi(V)+2 
	\leq c-\sum_{i=1}^{\mu} \chi(\Sigma_i) 
	\end{equation}
for all $\omega\in \mathbb{T}_!^\mu$.
\end{theorem}

The proof will follow along arguments from~\cite[Section~3]{ConwayNagelToffoli}. 
We will point out the necessary adaptions, but suppress arguments which
go through verbatim.
\begin{proof}
As in Lemma~\ref{lem:FreelyGenerated}, denote the exterior of $\Sigma$ in $W = (S^3 \times I) \csum V$ by $W_\Sigma$. 
Observe that parts of the boundary of $W_\Sigma$ are identified with the exteriors~$-X_L$ and~$X_{L'}$.
Use Lemma~\ref{lem:FreelyGenerated} to pick a homomorphism~$H_1(W_\Sigma; \Z) \to \Z^\mu$ that sends the meridians of $\Sigma_i$ to the canonical basis element~$e_i$ of $\Z^\mu$. 

The idea of the proof is to bound the twisted signature of $W_\Sigma$ using its twisted Betti numbers. 
The corresponding twisted signatures are related to the multivariable signatures of~$L$ and $L'$, while the twisted Betti numbers are related to the multivariable nullities $\eta_L(\omega),\eta_{L'}(\omega)$, as well as to $\chi(\Sigma_i)$ and $\chi(V)$.

Proceeding word by word as in~\cite[Proof of Lemma~3.10]{ConwayNagelToffoli}, 
establish that the homology groups decompose as follows: 
\[ H_i(\partial W_\Sigma; \C^\omega)\cong H_i(X_L; \C^\omega) \oplus H_i(X_{L'}; \C^\omega) .\] 
Deduce the following about the twisted Betti numbers~$\beta_i^\omega$ of~$W_\Sigma$:  
\[ \beta_1^\omega(W_\Sigma) \leq \eta_L(\omega) \text{ and } 
\beta_1^\omega(W_\Sigma) \leq \eta_{L'}(\omega),\]
and $H_i(W_\Sigma; \C^\omega) = 0$ for $i = 0,3,4$.
To see this, use that $\eta_L(\omega)=\dim_\C H_i(X_L; \C^\omega)$ (and similarly for $L'$) as well as Poincar\'e duality and the universal coefficient theorem; see \cite[Proof of Lemma~3.11]{ConwayNagelToffoli} for further details.

Next, we estimate the dimension of the space that supports the intersection form of $W_\Sigma$.
Let $j \colon H_2(\partial W_\Sigma; \C^\omega) \to H_2(W_\Sigma; \C^\omega)$ be the inclusion induced map. 
The twisted intersection form descends to~$\op{coker} (j)$.
The dimension of this space is~$\beta_2^\omega(W_\Sigma) - \beta_2^\omega(\partial W_\Sigma) + \beta_1^\omega(W_\Sigma)$ \cite[Proof of Lemma~3.12]{ConwayNagelToffoli}.
The Euler characteristic is the alternating sum of the Betti numbers.
Since the Euler characteristic can be computed with any coefficients, we obtain the following estimate (again details can be found in~\cite[Proof of Theorem~3.7]{ConwayNagelToffoli}):
\begin{equation} \label{eq:CNT}
|\text{sign}_\omega({W_\Sigma})|+|\eta_{L'}(\omega)-\eta_{L}(\omega)| \leq \chi(W_\Sigma).
\end{equation}

Our goal is to compute the twisted signature $\text{sign}_\omega({W_\Sigma})$. Let $F \subset D^4$ be a colored bounding surface for $L$ and let $\nu(F)$ be a tubular neighborhood for $F$ in~$D^4$. 
Recall that~$\sigma_{L}(\omega) = \text{sign}_\omega(D^4 \setminus \nu(F))$. Glue $D^4$ to $W$ by identifying $S^3 = \partial D^4$ with $-\big(S^3 \times \{-1\}\big) \subset \partial W$, which contains $L$. Call the result~$Z$.
The surfaces~$F_i$ and~$\Sigma_i$ glue together to~$S_i = F_i \cup_{L_i \times S^1} \Sigma_i \subset Z$, and we set $S:=S_1 \cup \ldots \cup S_\mu$.
Denote the exterior of~$S$ by $Z_S$, which naturally decomposes as~$Z_S = \big( D^4 \setminus \nu(F) \big) \cup W_\Sigma$.
Since $\omega \in \mathbb{T}^\mu$, the additivity of signatures \cite[Proposition~2.13]{ConwayNagelToffoli} implies that 
\begin{align}\label{eq:AdditivityTwistedZW}
 \text{sign}_\omega(Z_S)&=\sigma_{L}(\omega)+\text{sign}_\omega(W_\Sigma), \\
\text{sign}(Z_S)
& =\text{sign}(W_\Sigma). 
\nonumber
 \end{align}
For the second equality, we used the additivity of signatures together with the fact that $\sign(D^4 \setminus \nu F)=0$~\cite[Proposition 3.3]{ConwayNagelToffoli}.
Note that $H_1(Z;\Z)$ vanishes, $\partial  Z = S^3$, the link~$L'$ bounds~$S \subset Z$, and the $[S_i,\partial S_i] \in H_2(Z, \partial Z;\Z)$ vanish; see e.g. Lemma~\ref{lem:StablyConcordantUnlinkStablySlice}.
We are therefore in the setting of Theorem~\ref{thm:OtherAmbientSpaces} and, consequently, we obtain 
\[ \sigma_{L'}(\omega)
=\text{sign}_{\omega}(Z_S)-\text{sign}(Z_S). \] 
Combining this with~\eqref{eq:AdditivityTwistedZW}, we therefore deduce that
\begin{align}\label{eq:DFL}
 \text{sign}_\omega(W_\Sigma)
&=\text{sign}_\omega(Z_S)-\sigma_{L}(\omega)
 =\sigma_{L'}(\omega)-\sigma_{L}(\omega)+\text{sign}(W_\Sigma) \\ \nonumber
 &= \sigma_{L'}(\omega)-\sigma_{L}(\omega)+\text{sign}(W), 
 \end{align}
where Theorem~\ref{thm:OtherAmbientSpaces} is used for the last equality.

We are now in position to conclude. We remark that $\chi(W_\Sigma)=\chi(W)-\chi(\Sigma)$. 
Combining this observation with Equation~(\ref{eq:CNT}) and Equation~(\ref{eq:DFL}), it follows that 
\begin{equation}
\label{eq:NearlyGenusBound}
|\sigma_{L'}(\omega)-\sigma_{L}(\omega)+\text{sign}(W)|+|\eta_{L'}(\omega)-\eta_{L}(\omega)| \leq \chi(W)-\chi(\Sigma).
\end{equation}
Since $W$ was defined as $V \csum (S^3 \times I)$, we deduce that $\chi(W)=\chi(V)-2$. An Euler characteristic computation now shows that $\chi(\Sigma)=\sum_{i=1}^\mu \chi(\Sigma_i)-c$; see \cite[Proof of Lemma 3.9]{ConwayNagelToffoli}.
Inserting this into~\eqref{eq:NearlyGenusBound} yields~\eqref{eq:ImprovedGenusBound} and the proof is completed.
\end{proof}

Note that when $V=S^4$, Theorem~\ref{thm:ImprovedGenusBound} recovers~\cite[Theorem 3.7]{ConwayNagelToffoli}, which is itself a generalization of the classical Murasugi-Tristram inequality~\cite{Murasugi,Tristram, FlorensGilmer, Florens, CimasoniFlorens, Viro09, Powell}. Next, we verify the formula of Theorem~\ref{thm:ImprovedGenusBound} for $V = \overline { \C P}^2$ and highlight its connection to crossing changes.
\begin{example}
\label{ex:CP2}
	Suppose a link $L'$ is obtained from a link $L$ by changing a positive crossing within a component of $L$ to a negative crossing. 
    The trace of the isotopy
depicted in Figure~\ref{fig:ComplexHandle} forms a nullhomologous concordance from~$L$ to~$L'$ in~$(S^3 \times I) \csum \overline{\C P}^2$. 
Consider the special case of the right-handed trefoil~$L$, whose signature is~$\sigma_{L}(-1) = -2$. By changing a positive crossing, we obtain the unknot~$L'= U$ with signature~$\sigma_U(-1) =~0$. We verify the formula of Theorem~\ref{thm:ImprovedGenusBound} by noting that~$| 0 - (-2) - 1 | + 0 \leq 1$.
\begin{figure}
\begin{tikzpicture}
	\node at (0,0) {\includegraphics{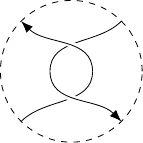}};
	\node at (1.75,0) {$\rightsquigarrow$};
	\node at (3.5,0) {\includegraphics{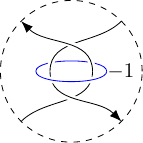}};
	\node at (5.5,0) {$\rightsquigarrow$};
	\node at (8.5,0) {\includegraphics{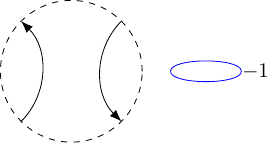}};
\end{tikzpicture}
	\caption{Performing a crossing change by sliding in $\overline{\C P}^2$.}
\label{fig:ComplexHandle}
\end{figure}
\end{example}

We now specialize Theorem~\ref{thm:ImprovedGenusBound} to stably slice links by setting $V = n S^2 \times S^2$. As a result, we obtain Theorem~\ref{thm:S2xS2Intro} from the introduction.
\begin{theorem}\label{thm:S2xS2}
If an $m$--component link $L$  is stably slice, then
\[ |\sigma_L(\omega)|+|\eta_L(\omega)-m+1| \leq 2\sn(L) .\]
for all $\omega \in \mathbb{T}^m_!$.
\end{theorem}
\begin{proof}
Assume that $L$ bounds $m$ disjoint nullhomologous slice disks in $D^4 \csum nS^2 \times~S^2$, with~$n=\op{sn}(L)$.
By taking $V=n\, S^2 \times S^2$ and 
$L'$ to be the unlink in Theorem~\ref{thm:ImprovedGenusBound},
we obtain the following inequality:
\begin{equation}\label{eq:ProofS2xS2}
 |\sigma_L(\omega)+\text{sign}(n\, S^2 \times S^2)|+|\eta_L(\omega)-m+1| -\chi(n\,S^2 \times S^2)+2 \leq c-\sum_{i=1}^{m} \beta_1(F_i).
\end{equation}
Since $L$ is stably slice, the right hand side of (\ref{eq:ProofS2xS2})
is zero. Since the signature is additive under direct sum and the intersection form of $S^2 \times S^2$ is hyperbolic, $\text{sign}(n\,S^2 \times S^2)$ also vanishes. Since the Euler characteristic of $S^2 \times S^2$ is $4$, we get $\chi(n\,S^2 \times S^2)=4n-2(n-1)=2n+2$. The result follows.
\end{proof}

Next, we provide an example of Theorem~\ref{thm:S2xS2}.
\begin{example}\label{ex:6Compo} 
We show that the $6$--component link $L$ in Figure~\ref{fig:6Compo} has~$\sn(L) =~2$.
When arguing as in Remark~\ref{rem:Kirby}, we may arrange the $0$-framed Hopf link so that each of its components links multiple bands. 
Using such a \emph{generalized band pass}, we can move a single band passed multiple bands. 
Via such a move, at the cost of a single $S^2 \times S^2$, we can split off two components of $L$ (corresponding to the Bing double of one component of the Borromean pattern).
After another band pass move, we obtain the $6$-component unlink, and so~$\sn(L) \leq 2$.
The equality is obtained using Theorem~\ref{thm:S2xS2}.
To compute the multivariable signature and nullity of~$L$, we use C-complexes~\cite{CimasoniFlorens}.
Pick the obvious planar $3$--component C-complex~$F$ for $L$, all of whose surfaces are disks.
A short computation shows that generalized Seifert matrices for $L$ are given by  $A^\varepsilon= \bsm 0 & 1&1\\ 1&0 &1 \\ 1 & 1 &0 \esm$ for each~$\varepsilon$. It follows that $\sigma_L(-1,\ldots,-1)=-1$ and $\eta_L(-1,\ldots,-1)=2$, since $F$ has $3$ components; see~\cite[Section 2]{CimasoniFlorens}. 
Consequently, the bound of Theorem~\ref{thm:S2xS2} is equal to
\[  |\sigma_L(\omega)|+|\eta_L(\omega)-m+1|=|-1|+|2-6+1|=4.\]
This shows that $\operatorname{sn}(L)=2$, as claimed.
\begin{figure}[!htb]
\includegraphics{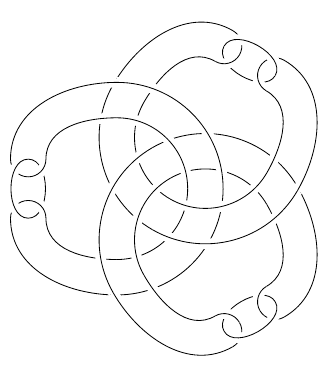}
\caption{The $6$--component link of Example~\ref{ex:6Compo}.}
\label{fig:6Compo}
\end{figure}
\end{example}

For the sake of completeness, we also provide a family of knots with arbitrary stabilizing number.
\begin{example}
We assert that the knot~$K_\ell = \#_{i=1}^\ell \, 8_7$, which is an $\ell$--fold connected sum of the knot~$K=8_7$ has stabilizing number~$\sn(K_\ell) = \ell$.
First note that since $\Arf(K)=0$, we have $\Arf(K_\ell)=0$ and therefore $\sn(K_\ell) < \infty$ for each $\ell$.
As $\sigma_K(-1)=2$, we deduce from Theorem~\ref{thm:S2xS2} that $ 2\ell=\sigma_{K_\ell}(-1) \leq 2\sn(K_\ell)$.
Theorem~\ref{thm:StablGenus}, which we prove in Section~\ref{sec:4Genus}, gives $ \sn(K_\ell) \leq g_4^{\op{top}}(K_\ell)$.
We obtain
$$ \ell=\frac{1}{2}\sigma_{K_\ell}(-1) \leq \sn(K_\ell)\leq g_4^{\op{top}}(K_\ell) \leq \ell g_4^{\op{top}}(K)=\ell,$$
where in the last equality, we used that $g_4(K_\ell)=1$; see~\cite{KnotInfo}.
\end{example}

\section{Casson-Gordon invariants}\label{sec:CG}
We give a bound on the stabilizing number in terms of Casson-Gordon invariants in Theorem~\ref{thm:CGstabilizing}, and an example of an algebraically slice knot with non-zero stabilizing number.
 Furthermore, we construct knots~$K$ such that $\op{sn}(K) < g_4^{\op{top}}(K)$.
 This section is organized as follows. In Section~\ref{sub:CGDef}, we review Casson-Gordon invariants and in Section~\ref{sub:CGstabilizing}, we prove Theorem~\ref{thm:CGstabilizing}.
 In Section~\ref{sub:NotGenus}, we describe an infinite family of knots that have $\sn(K)=1$ and $g_4^{\op{top}}(K)=2$.

\subsection{Background on Casson-Gordon invariants}\label{sub:CGDef}

Given a knot $K$, we use the following notations:~$M_K$ is the $0$--framed surgery of $S^3$ along $K$. Its $n$--fold cover is denoted by $p \colon M_n(K) \to M_K$, and $\Sigma_n(K)$ is the $n$--fold cover over $S^3$ branched along $K$. For the next paragraphs, we fix an epimorphism
\[ \chi \colon H_1(\Sigma_n(K);\Z) \to \Z_d. \]
For the remainder of this subsection, denote the $d$--th root of unity by $\omega=\operatorname{exp}(\frac{2\pi i}{d})$ 

Since the bordism group~$\Omega_3(\Z_d)$ is finite, there exists a non-negative integer~$r$, a $4$--manifold $W$ and a map $\psi \colon \pi_1(W) \to \Z_d$ such that $\partial (W,\psi) = r(\Sigma_n(K),\chi)$. The morphism 
$\Z[\pi_1(W)] \to \C$, defined by $g \mapsto \omega^{\psi(g)}$,
gives rise to twisted homology groups $H_*(W;\C^\psi)$ and to a $\C$--valued Hermitian intersection form~$\lambda_{\C}$ on~$H_2(W;\C^\psi)$, whose signature is denoted $\text{sign}^\psi(W):= \text{sign}(\lambda_{\C}(W))$. 
On the boundary, we obtain twisted homology groups $H_*(\Sigma_n(K),\C^\chi)$.
\begin{definition}\label{def:CassonGordonSigma}
Let $K$ be a knot and let $\chi \colon H_1(\Sigma_n(K);\Z) \to \Z_d$ be an epimorphism.
 The \emph{Casson-Gordon $\sigma$--invariant} and \emph{Casson-Gordon nullity} are
\begin{align*}
& \sigma_n(K,\chi):=\frac{1}{r}\left (\text{sign}^\psi(W)-\text{sign}(W)\right ) \in \mathbb{Q}. \\
	&\eta_n(K,\chi):=\dim_\C H_1(\Sigma_n(K);\C^\chi).
\end{align*}
\end{definition}

Casson and Gordon showed that $\sigma_n(K,\chi)$ is well-defined~\cite[Lemma 2.1]{CassonGordon1}.
Note that it is sometimes convenient to think of $\chi$ as a (not necessarily surjective) character with values in some~$\Z_m$ with order $d$; see also Remark~\ref{rem:CharacterLinkingForm}.
Analogous signature and nullity invariants are also defined
 for the more general setting described below.
\begin{remark}\label{rem:Sigma3Manifold}
Given a closed $3$--manifold with an epimorphism $\chi \colon H_1(M;\Z) \to~\Z_d$, define a signature invariant $\sigma(M,\chi)$ just as in Definition~\ref{def:CassonGordonSigma}: bordism theory ensures the existence of a pair $(W,\psi)$ such that $\partial (W,\psi)=r(M,\psi)$ and declare
\[ \sigma(M,\chi):=\frac{1}{r}\left (\text{sign}^\psi(W)-\text{sign}(W)\right ). \]
\end{remark}

Next, we define the invariant $\tau_n(K,\chi)$. 
Composing the projection-induced map with abelianization gives rise to the map
\[ \pi_1(M_n(K)) \xrightarrow{p_*} \pi_1(M_K) \xrightarrow{\phi_K} H_1(M_K;\Z) \cong \Z.\] 
Since the image of this map is isomorphic to $n \mathbb{Z}$, mapping to it produces a surjective map $\alpha \colon \pi_1(M_n) \twoheadrightarrow n \mathbb{Z}$. 
Two Mayer-Vietoris arguments provide a surjection $H_1(M_n(K);\Z)\cong H_1(X_n(K);\Z) \twoheadrightarrow H_1(\Sigma_n(K);\Z)$.
Using this surjection, we deduce that
the character~$\chi$ induces a character on $H_1(M_n(K);\Z)$ for which we use the same notation. We have therefore obtained a homomorphism
\begin{align*}
\alpha \times \chi\colon \pi_1(M_n(K)) &\stackrel{}{\longrightarrow} \mathbb{Z} \times \mathbb{Z}_d\\
 g  &\mapsto (t_K^{n\phi_K(g)},\chi(g)).
\end{align*}
As the bordism group $\Omega_3(\Z \times \Z_d)=H_3(\Z \times \Z_d;\Z)=H_3(\Z_d;\Z)$ is finite, there is a non-negative integer~$r$, a $4$--manifold $W$ and a homomorphism $\psi \colon \pi_1(W) \to \Z \times \Z_d$ such that~$\partial (W,\psi)=~r(M,\alpha \times~\chi)$. 
The assignment $(k,l) \to \omega^l t^k$ gives rise to a morphism $\Z[\Z \times \Z_d] \to \C(t)$. The composition 
\begin{align*}
\Z[\pi_1(W)] \xrightarrow{\psi} \Z[\Z \times \Z_d] &\to \C(t)
\end{align*}
gives rise to twisted homology groups $H_*(W;\C(t)^\psi)$. 
Assume from now on that $d$ is a prime power.
Using~\cite[Lemma 4]{CassonGordon2}, this implies that~$H_*(M_n(K);\C(t)^\psi) =~0$.
Therefore $\lambda_{\C(t)}$ is non-singular and defines a Witt class $[\lambda_{W,\C(t)}] \in W(\C(t))$. 
Whereas the standard intersection pairing $\lambda_{W,\Q}$ may be singular, modding out by the radical leads to a nonsingular form~$\lambda_{W,\Q}^{\text{nonsing}}$. Using the inclusion induced map $W(\mathbb{Q}) \to W(\C(t))$, we therefore obtain $[\lambda_{W,\Q}^{\text{nonsing}}]$ in $W(\C(t))$.

\begin{definition}\label{def:CassonGordonInvariant}
Let $K$ be an oriented knot, let $n,m$ be positive integers and let~$\chi \colon H_1(\Sigma_n(K);\Z) \to \Z_m$ be a prime power order character.
The \emph{Casson-Gordon $\tau$--invariant} is the Witt class
\[ \tau_n(K,\chi)=\Big([\lambda_{W,\mathbb{C}(t)}]-[\lambda_{W,\Q}^{\text{nonsing}}] \Big) \otimes \frac{1}{r} \in  W(\mathbb{C}(t))\otimes \mathbb{Q}. \]
\end{definition}

Note that if $n$ is additionally assumed to be a prime power, then $\tau_n(K,\chi)$ provides an obstruction to sliceness~\cite[Theorem 2]{CassonGordon2}.
In what follows, we focus on the case $n=2$ and abbreviate $\sigma_2(K,\chi),\eta_2(K,\chi)$ and $\tau_2(K,\chi)$ as $\sigma(K,\chi),\eta(K,\chi)$ and $\tau(K,\chi)$.

The next remark reviews how the linking form on $H_1(\Sigma_n(K);\Z)$ provides a convenient way to list the finite order characters on $H_1(\Sigma_n(K);\Z)$.  Here, recall that given a rational homology sphere $M$, the linking form $\beta_M$ is a $\Q/\Z$-valued non-singular symmetric pairing on $H_1(M;\Z)$. We write $\beta_K$ instead of $\beta_{\Sigma_2(K)}$.

\begin{remark}\label{rem:CharacterLinkingForm}
Think of the group~$\Z_m$ as the quotient~$\Z/m\Z$, and note that $k \mapsto \frac{k}{m}$ defines an isomorphism onto its image in~$\Q/\Z$.
Using this isomorphism, every character $\chi \colon H_1(M;\Z) \to \Z_m$  defines an element~$\chi/m$ in the group $\op{Hom}_\Z( H_1(M;\Z),\Q/\Z )$. 
Conversely, any homomorphism $H_1(M;\Z) \to \Q/\Z$ is of finite order, say $m$, and so is of the form~$\chi/m$ for a character~$\chi \colon H_1(M; \Z) \to \Z_m$. 
Thus such finite order characters correspond bijectively to elements of $\op{Hom}_\Z(H_1(M;\Z),\Q/\Z)$. 
If $M$ is a rational homology sphere, then the linking form $\beta_M$ is nonsingular, and its adjoint defines an isomorphism
$H_1(M; \Z) \xrightarrow{\sim} \op{Hom}_\Z(H_1(M;\Z),\Q/\Z)$. 
Via these two correspondences, we associate to an element~$x \in H_1(M;\Z)$ a character~$\chi_x \colon H_1(M; \Z) \to \Z_m$ for some $m$, or immediately refer to elements~$\chi \in H_1(M;\Z)$ as characters.
\end{remark}

Just as in Remark~\ref{rem:CharacterLinkingForm}, we shall often think of characters on $H_1(M;\Z)$ as taking values in $\Q/\Z$, the choice of $m$ being somewhat immaterial.
Motivated by Remark~\ref{rem:CharacterLinkingForm}, we also recall some terminology on linking forms. If $B$ is a non-degenerate size $n$ symmetric matrix over $\Z$, then we can consider the non-singular symmetric pairing
\begin{align*}
 \lambda_B \colon \Z^n/B\Z^n \times  \Z^n/B\Z^n &\to \Q/\Z \\
(x,y) &\mapsto x^T B^{-1}y. 
\end{align*}

We say that $\beta_K$ is \emph{presented by} a non-degenerate symmetric matrix $B$ if $\beta_K$ is isometric to $\lambda_B$. 
\begin{remark} \label{rem:LinkingFormSeifert}
Given a knot $K$, the linking form $\beta_K$ is presented by $A+A^T$, where~$A$ is any Seifert matrix for $K$~\cite[p. 31]{GordonSurvey}.
\end{remark}

We conclude this subsection by mentioning a satellite formula for $\sigma$ and $\eta$ in the winding number $0$ case. Such a formula is stated in~\cite{Abchir}, but as we illustrate in Example~\ref{ex:Motivation} below, it unfortunately contains a mistake. While Appendix~\ref{Appendix}, contains a suitably modified statement, the following result is all we need for the moment.

\begin{theorem}\label{thm:AbchirSimplified}
Let $R(J,{\gamma})$ be a satellite with pattern $R$, companion $J$, infection curve ${\gamma}$ and winding number $0$. Let $\chi \colon H_1(\Sigma_2(R(J,{\gamma}));\Z) \to \Q/\Z$ be a character of prime power order $d$ and set $\omega:=e^{2\pi i/d}$. If we use ${\gamma}_1,{\gamma}_2$ to denote the lifts of~${\gamma}$ to the $2$--fold branched cover $\Sigma_2(R(J,{\gamma}))$, then
\begin{align*}
& \sigma(R(J,{\gamma}),\chi)=\sigma(R,\chi_R)+\left( \sigma_J(\omega^{\chi({\gamma}_1)})+\sigma_J(\omega^{\chi({\gamma}_2)})\right), \\
& \eta(R(J,{\gamma}),\chi)=\eta(R,\chi_R)+\left( \eta_J(\omega^{\chi({\gamma}_1)})+\eta_J(\omega^{\chi({\gamma}_2)})\right).
\end{align*}
Here, $\chi_R$ denotes the character on $H_1(\Sigma_2(R);\Z)$ induced by $\chi$; see Lemma~\ref{lem:CorrespondenceCharacters} for further details.
Furthermore, on connected sums, we have
\begin{align*}
& \sigma(K_1 \csum K_2,\chi_1 \oplus \chi_2)=\sigma(K_1,\chi_1)+\sigma(K_2,\chi_2),\\
& \eta(K_1 \csum K_2,\chi_1 \oplus \chi_2)
= \begin{cases}
\eta(K_1,\chi_1)+\eta(K_2,\chi_2) \quad &\text{ if one of the } \chi_i \text{ is trivial,} \\
\eta(K_1,\chi_1)+\eta(K_2,\chi_2)+1 \quad &\text{ if neither } \chi_i \text{ is trivial.}
\end{cases}
\end{align*}
\end{theorem}
\begin{proof}
The proof of the winding number zero satellite formula for the $\sigma$-invariant can be found in Corollary~\ref{cor:Abchir0Modn} below.
The satellite formula for the nullity is proved in Proposition~\ref{prop:CGNullitySatellite} below.
The connected sum formulas are known~\cite[Proposition~2.5]{FlorensGilmer}.
\end{proof}

\subsection{Casson-Gordon invariants and the stabilizing number} \label{sub:CGstabilizing}
The aim of this subsection is to use Casson-Gordon invariants in order to provide an obstruction to a knot having a given stabilizing number. The result and its proof resemble a result of Gilmer~\cite[Theorem 1]{GilmerGenus}, which we also need in Section~\ref{sub:NotGenus}. Gilmer's original theorem only makes use of the Casson-Gordon $\tau$ invariant; the following reformulation appears in~\cite[Theorem 4.1]{FlorensGilmer}.

\begin{theorem}[Gilmer]\label{thm:CGGenus}
If a knot $K$ bounds a genus $g$ locally flat embedded surface~$F \subset D^4$, then the linking form $\beta_K$ decomposes as a direct sum $\beta_1 \oplus \beta_2$ such that the following two conditions hold:
\begin{enumerate}
\item the linking form $\beta_1$ has an even presentation matrix of rank $2g$ and signature~$\sigma_K(-1)$;
\item there is a metabolizer $G$ of $\beta_2$ such that for all prime power characters $\chi \in G$, we have
\begin{equation}
\label{eq:GenusIneqThm}
 |\sigma(K,\chi)+\sigma_K(-1)| \leq \eta(K,\chi)+4g+1.
 \end{equation}
\end{enumerate}
\end{theorem}

In order to obtain the corresponding result for stabilizing numbers,  we start by stating four lemmas: the two first of which are due to Gilmer. 
The first lemma is crucial to extending characters on a $3$--manifold to a bounding $4$--manifold.

\begin{lemma}\label{lem:SplitLinkingForm}
If $M$ bounds a spin $4$--manifold $W$, then $\beta_M=\beta_1 \oplus \beta_2$, where~$\beta_2$ is metabolic and $\beta_1$ has an even presentation matrix of
rank $\dim_\Q H_2(W;\Q)$ and signature $\sign(W)$. Moreover, the set of characters that extend to $H_1(W;\Z)$ forms a metabolizer of $\beta_2$.
\end{lemma}
\begin{proof}
See~\cite[Lemma~1]{GilmerGenus}.
\end{proof}

The second lemma ensures that the first homology group of certain infinite cyclic covers remains finite dimensional.
\begin{lemma}\label{lem:FiniteDimensional}
Let $X$ be a connected infinite cyclic cover of a finite complex $Y$ and~$\widetilde{X}$ a $p^r$ cyclic cover of $X$ for a prime $p$. If $H_k(Y;\Z)=0$, then $H_k(\widetilde{X};\Q)$ is finite dimensional. If $H_1(Y;\Z)=\Z$, then $H_1(\widetilde{X},\Q)$ is finite dimensional.
\end{lemma}
\begin{proof}
See~\cite[Lemma 2]{GilmerGenus}.
\end{proof}

The third lemma ensures that we will be able to apply Lemma~\ref{lem:SplitLinkingForm} in the context of stabilizing numbers. 
\begin{lemma}\label{lem:Spin}
Let $D$ be a properly embedded locally flat  disk in~$D^4 \csum n S^2 \times S^2$ and let $\Sigma_2(D)$ denote the $2$-fold cover of $D^4 \csum n S^2 \times S^2$ branched along $D$. If $D$ is nullhomologous, then the branched cover $\Sigma_2(D)$ is spin.
\end{lemma}
\begin{proof}
Since $\Sigma_2(D)$ is oriented, it is enough to show that the second Stiefel-Whitney class $w_2(\Sigma_2(D))$ vanishes.
Note that $n S^2 \times S^2$ is spin, since it is simply connected and
has even intersection form. Consequently, the manifold~$D^4 \csum nS^2 \times~S^2$ is also spin, and its second Stiefel-Whitney class vanishes:
\begin{equation}\label{eq:w2S2S2}
w_2(D^4 \csum n S^2 \times S^2)=0.
\end{equation}
Next, consider the two-fold branched cover~$\pi \colon \Sigma_2(D) \to D^4 \csum n S^2 \times S^2$. 
Since $D$ is nullhomologous, a result due to Gilmer~\cite[Theorem~7]{GilmerSignatures} ensures the existence of
a complex line bundle~$E$ with Chern class $c_1(E) =~0$ such that
\begin{equation}
\label{eq:BundleForw2}
T \Sigma_2(D) \oplus E^d = \pi^* \big( T (D^4 \csum n S^2 \times S^2) \oplus E\big).
\end{equation}
Note that~$w_2(E) = c_1(E)=0 \mod 2$ and~$w_1(E) = 0$ since~$E$ is orientable. This implies that both $w_2(E^d)$ and $w_1(E^d)$ vanish. 
Using the Whitney product formula and the naturality of the characteristic classes, in~\eqref{eq:BundleForw2}, and using~\eqref{eq:w2S2S2} yields
\[ w_2(\Sigma_2(D)) + 0 + w_1( \Sigma_2(D) )\cup 0 = \pi^*\big( w_2(D^4 \csum n S^2 \times S^2) + 0 + 0 \big) = 0.\]
Thus~$w_2(\Sigma_2(D)) = 0$ and therefore $\Sigma_2(D)$ is spin, as announced.
\end{proof}

The next lemma computes the second Betti number of $\Sigma_2(D)$.
\begin{lemma}\label{lem:SndHomology}
If $D$ is a nullhomologous properly embedded disk in~$D^4 \csum n S^2 \times S^2$, then
	\[\dim_\Q H_2(\Sigma_2(D);\Q)=4n.\]
\end{lemma}
\begin{proof}
This is a homological calculation.
Denote the $2$--fold unbranched cover of the disk exterior~$W_D$ by~$W_2(D)$.
The action of the deck transformation group of~$W_2(D)$ induces an automorphism of order~$2$ on $H_k(W_2(D);\C)$. 
Decompose~$H_k(W_2(D);\C)$ into the corresponding eigenspaces~$\op{Eig}(H_k(W_2(D);\C), \pm 1)$. The resulting Betti numbers are denoted $b_k^{\pm}(W_2(D))$. The idea of the proof is to compute $b_2^{\pm}(W_2(D))$ in order to deduce the value of $b_2(W_2(D))$ and $b_2(\Sigma_2(D))$.

For $\omega=\pm 1$, we saw in Lemma~\ref{lem:Eigenspaces} that the $\omega$--eigenspace $\op{Eig}(H_k(W_2(D);\C), \omega)$ is isomorphic to the vector space $H_k(W_D;\C^\omega)$. For $\omega=1$, observe that $\C^1$ is the~$\C[\Z_2]$--module~$\C$ with the trivial~$\Z_2$--action and therefore
\[ \op{Eig}(H_k(W_2(D); \C), +1)  \cong H_k(W_2(D); \C^{1}) \cong H_k(W_D; \C).  \]
In particular, we deduce that $b_2^{+}(W_2(D))=b_2(W_D)=2n$. 

Two estimates of Gilmer~\cite[Propositions 1.4 and 1.5]{GilmerConfiguration} imply that
\begin{align} \label{eq:Ineq}
	b_2^{-}(W_2(D)) &\leq \dim_{\Z_2}H_2(W_D;\Z_2)= 2n, \\
	b_1^{-}(W_2(D)) &\leq \dim_{\Z_2}H_1(W_D;\Z_2)-1= 0. \nonumber
\end{align}

Next, denote the alternating sum of the Betti numbers~$b_k^{\pm}$ by $\chi^{\pm}$.
An application of~\cite[Proposition 1.1]{GilmerConfiguration} ensures that 
\begin{equation}\label{eq:EulerCharc}
\chi^{-}(W_2(D))=\chi^{+}(W_2(D))=\chi(W_D) = 1-1+2n=2n. 
\end{equation}
Since~$b_0^{+}(W_2(D)) = b_0(W_D) = 1$, one sees that $b_0^{-}(W_2(D))=0$.
Analogously, we show that~$b_0^{-}(\partial W_2(D)) = 0$.
Using~\eqref{eq:Ineq}, we know that $b_1^{-}(W_2(D))=0$, and this implies that~$b_3^-(W_2(D))=0$: the long exact sequence gives $b_1^{-}(W_2(D),\partial W_2(D))=~0$, then use Poincar\'e duality and universal coefficents. 
Equation~\eqref{eq:EulerCharc} now implies that
\[ b_2^{-}(W_2(D))=\chi^{-}(W_2(D))=2n.\]
We therefore obtained that~$b_2(W_2(D)) = b_2^{+}(W_2(D))+b_2^{-}(W_2(D)) = 4n$.
The lemma follows from the Mayer-Vietoris sequence for $\Sigma_2(D) = W_2(D) \cup (D \times D^2)$. 
\end{proof}

The following theorem, which is Theorem~\ref{thm:CGstabilizingIntro} from the introduction, provides an obstruction to the stabilizing number that involves the Casson-Gordon invariants.
Its proof is similar to the proof of~\cite[Theorem~1]{GilmerGenus}; see also~\cite[Theorem~4.1]{FlorensGilmer}.
\begin{theorem}\label{thm:CGstabilizing}
If a knot $K$ bounds a locally flat embedded disk $D$ in~$D^4 \csum n S^2 \times~S^2$, then the linking form $\beta_K$ can be written as a direct sum $\beta_1 \oplus \beta_2$ such that
\begin{enumerate}
\item $\beta_1$ has an even presentation matrix of rank $4n$ and signature $\sigma_K(-1)$;
\item there is a metabolizer of $\beta_2$ such that for all characters of prime power order in this metabolizer, 
\[ |\sigma(K,\chi)+\sigma_K(-1)| \leq \eta(K,\chi)+4n+1. \]
\end{enumerate}
\end{theorem}

\begin{remark}\label{rem:WeakerThan4Genus}
The obstruction in Theorem~\ref{thm:CGstabilizing} is weaker than the corresponding result for the $4$--genus stated in Theorem~\ref{thm:CGGenus}.
In Theorem~\ref{thm:CGstabilizing}, the form $\beta_1$ is presented by a matrix of rank $4n$, 
while in the $4$--genus estimate, the form~$\beta_1$ is presented by a matrix of rank $2g$.
\end{remark}

\begin{proof}[Proof of Theorem~\ref{thm:CGstabilizing}]
Assume that $K$ bounds a nullhomologous embedded disk~$D$ in the stabilized $4$--ball $D^4 \csum nS^2 \times~S^2$, 
and let $W_D$ denote its exterior. 
Let $\Sigma_2(D)$ be the corresponding $2$--fold branched cover. 

Lemma~\ref{lem:Spin} shows that~$W$ is spin, and Lemma~\ref{lem:SndHomology} computes that~$H_2(W;\Q)$ has dimension~$4n$.
We apply Lemma~\ref{lem:SplitLinkingForm} to $M=\Sigma_2(K)$ with the bounding $4$--manifold being $W=~\Sigma_2(D)$ to obtain the splitting of the linking form~$\beta_K$. 
Deduce that the linking form $\beta_K$ splits as a direct sum $\beta_1 \oplus \beta_2$, where $\beta_2$ is metabolic and~$\beta_1$ has an even presentation of rank $4n$ and signature~$\sign(\Sigma_2(D))$. 

Let $G$ denote the metabolizer of $\beta_2$ that consists of those characters of the group~$H_1(\Sigma_2(K);\Z)$ that extend to characters of $H_1(\Sigma_2(D);\Z)$. For such a character~$\chi \in G$, in the remainder of this proof, we will establish that 
\begin{equation}\label{eq:GoalCG}
	|\sigma(K,\chi)+\sigma_K(-1)| \leq \eta(K,\chi)+4n+1.
\end{equation}
Let $M_2(K)$ denote the $2$--fold cover of 
$M_K$, the $0$--surgery of~$K$. 
Recall that the character $\chi$ on $H_1(\Sigma_2(K);\Z)$ gives rise to a character on $H_1(M_2(K);\Z)$.  
This was used to define the 
Witt class $\tau(K,\chi)$. Representing $\tau(K,\chi)$ by a Hermitian matrix~$A(t)$, setting $t=1$, and taking the averaged signature gives rise to a rational number $\sign_1^{\op{av}}(\tau(K,\chi))$; see~\cite[discussion preceding Theorem 3]{CassonGordon2} for details.
 Using consecutively the triangle inequality and an estimate~\cite[Theorem 3]{CassonGordon2}, 
we obtain
\begin{align*}
 |\sigma(K,\chi)+\sigma_K(-1)|&\leq |\sigma(K,\chi)-\sign_1^{\op{av}}(\tau(K,\chi))|+|\sign_1^{\op{av}}(\tau(K,\chi))+\sigma_K(-1)| \\
 &\leq  1+\eta(K,\chi)+|\sign_1^{\op{av}}(\tau(K,\chi))+\sigma_K(-1)|.
 \end{align*}
As a consequence,~\eqref{eq:GoalCG} will follow once we establish $|\sign_1^{\op{av}}(\tau(K,\chi))+\sigma_K(-1)| \leq~4n$.
Use $d$ to denote the order of $\chi$.
Recall that $\tau(K,\chi)=\tau(M_2(K),\alpha \times \chi)$ is 
defined as a difference of two Witt classes that arise from
a bounding $4$-manifold for $M_2(K)$ over which $\alpha \times \chi$ extends.

We claim that~$W_2(D)$ is such a bounding $4$--manifold.
Clearly, $M_2(K)$ bounds $W_2(D)$ and so we check that the representation extends. Since $\chi$ belongs to $G$, it extends to $H_1(\Sigma_2(D);\Z)$ and therefore to $H_1(W_2(D);\Z)$. On the other hand, recall that $\alpha \colon H_1(M_2(K);\Z) \to 2\Z \subset H_1(M_K;\Z) \cong \Z$ is the map induced by the covering projection, and is thus extended by the covering projection induced map~$H_1(W_2(D);\Z) \to H_1(W_D;\Z)$. 
We deduce that
\begin{equation} \label{eq:SignatureTau}
\sign_1^{\op{av}}\big(\tau(K,\chi)\big)=\sign_1^{\operatorname{av}} \lambda_{W_2(D),\C(t)}-\sign^{\operatorname{av}}_{1}\lambda_{W_2(D),\Q}.
\end{equation}
The signature $\sign^{\operatorname{av}}_1\lambda_{W_2(D),\Q}$ coincides with the signature $\sign(W_2(D))$ of $W_2(D)$.
The signature~$\sign_1^{\operatorname{av}} \lambda_{W_2(D),\C(t)}$ is bounded by the dimension of the $\C(t)$-vector space $H_2(W_2(D);\C(t)_{\alpha \times \chi})$, which we compute below. For this, denote the corresponding twisted Betti numbers by~$\beta_i^{\C(t)}(W_2(D))$, which are given by:
\begin{claim}
\label{claim1}
$\beta_i^{\C(t)}(W_2(D))=0$ for $i \neq 2$.
\end{claim}
Arguing as in~\cite[Lemma 8.1]{BorodzikConwayPolitarczyk}, the assertion holds for $i=0$. Using the long exact sequence, duality and universal coefficients, it only remains to show that $\beta_1^{\C(t)}(W_2(D))=0$. 
At this point, it is useful to think with covers. The map~$\alpha \times \chi$ induces a $(\Z \times \Z_d)$--cover $\widetilde{W}(D)$ of $W_2(D)$. In other words, we first get a cover $W_\infty(D) \to W_2(D)$ and then take the additional $d$--cover induced by the composition $H_1(W_\infty(D);\Z)\to H_1(W_2(D);\Z)\to \Z_d$. 
But now, $X=W_\infty(D)$ is also a cover of~$W_D$. Therefore, we can apply Lemma~\ref{lem:FiniteDimensional} to $Y=W_D$ to obtain that~$H_1(\widetilde{W}(D);\Q)$ is finite dimensional. Since $\C(t)$ is flat over $\Q[\Z_d \times \Z]$~\cite[p.~189]{CassonGordon2}, we get 
$H_{1}(W_2(D);\C(t)_{\alpha \times \chi}) =  \C(t) \otimes_{\Q[\Z_d \times \Z]} H_{1}(\widetilde{W}(D);\Q)=0$. 
This shows that~$\beta_1^{\C(t)}(W_2(D))=0$, and the claim is proved.

The Euler characteristic can be computed with any coefficients.
Therefore, the claim implies that $\beta_2^{\C(t)}(W_2(D))=\chi^{\C(t)}(W_2(D))$ is equal to the Euler characteristic~$\chi(W_2(D))=2\chi(W_D)=2(1-1+2n)=4n$ of the unbranched cover. We have established that $\sign_1^{\op{av}}\lambda_{W_2(D),\C(t)} \leq 4n$ and the theorem will follow promptly from the following claim:
\begin{claim} \label{claim2}
$\sign(W_2(D))=\sigma_K(-1)$.
\end{claim}
Recall that for $\omega \in \C$,  the twisted intersection form of $W_D$ with coefficient system $\C^\omega$ is denoted by $\lambda_{\C^\omega}(W_D)$. Taking $\omega=-1$, we know from Theorem~\ref{thm:OtherAmbientSpaces} that $\sigma_K(-1)=\sign\big(\lambda_{W_D,\C^{-1}}\big)$. 
Using Lemma~\ref{lem:Eigenspaces}, this twisted signature is the same as signature of the tensored up intersection form on $\C^{-1} \otimes_{\Z[\Z_2]} H_2(W_2(D);\Z)$.
This is the same as the signature of the intersection form on~$W_2(D)$ restricted to the $(-1)$--eigenspace of $H_2(W_2(D); \C)$. Since we are dealing with double covers, $\sign(W_2(D))$ is equal to the sum of the $(+1)$--signature and the $(-1)$--signature. 
The $(+1)$--signature is the signature of~$W_D$ and is therefore trivial.
The claim follows.

Using the claim, we know that the signature of the presentation matrix for $\beta_1$ 
is the signature $\sigma_K(-1)$. On the other hand, using~\eqref{eq:SignatureTau} and the claim, we deduce that $|\sign_1^{\op{av}}(\tau(K,\chi))+\sigma_K(-1)|=|\sign_1^{\op{av}}  \lambda_{W_2(D),\C(t)}|.$ We already saw above that this quantity is bounded by $4n$, and as we already mentioned, this was the last step to establish~\eqref{eq:GoalCG}. This concludes the proof of the theorem.
\end{proof}

The next example provides an application of Theorem~\ref{thm:CGstabilizing}.

\begin{example}\label{ex:AlgSlicesn2}
We construct an algebraically slice knot whose stabilizing number is at least $2$:
set $R := 9_{46}$ and pick a knot~$J$ with the following properties:
\begin{enumerate}
	\item the signature of $J$ satisfies $\sigma_J( e^{2\pi i /3} ) > \frac{3}{2}\big( |\sigma(R, \wh \chi)| + |\eta(R, \wh \chi')| \big) + \frac{7}{2}$ for all characters $\wh \chi, \wh \chi' \colon H_1(\Sigma_2(R);\Z) \to \Z_3$;
\item the Alexander polynomial~$\Delta_J$ is non-zero at third roots of unity. 
\end{enumerate}
\begin{figure}
\includegraphics{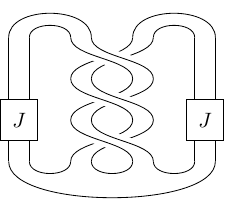}
\caption{The satellite knot $R(J,J)$ described in Example~\ref{ex:AlgSlicesn2}.}
\label{fig:KJ}
\end{figure}
Now consider the knot $K:=\#_{i=1}^3 R(J,J)$, where $R(J,J)$ is depicted in Figure~\ref{fig:KJ}.
Recall from Remark~\ref{rem:LinkingFormSeifert} that the linking form $\beta_K$ on $H_1(\Sigma_2(K);\Z)$ is presented by any symmetrized Seifert matrix for $K$.
A direct computation shows that the knot~$R$ admits $\bsm 0&1 \\ 2&0 \esm$ as a Seifert matrix; see e.g.~\cite[p. 325]{LivingstonSurvey}.
It follows that a Seifert matrix for $K$ is given by $\oplus_{i=1}^3 \bsm 0&1 \\2&0 \esm$.
Consequently, the matrix $\oplus_{i=1}^3 \bsm 0&3 \\3&0 \esm$ presents $\beta_K$, and $K$ is algebraically slice.
 We will argue that
\[ 2 \leq \operatorname{sn}(K)  \leq g_4^{\op{top}}(K)=3.\]
The inequality $\operatorname{sn}(K)  \leq g_4^{\op{top}}(K)$ is proved in Theorem~\ref{thm:StablGenus} below.
The equality $g_4^{\op{top}}(K)=3$ is known~\cite{GilmerGenus, FlorensGilmer}, but we will outline the argument below.
We show that $2 \leq \operatorname{sn}(K)$.
 By way of contradiction, assume that $\operatorname{sn}(K)=1$.
 Theorem~\ref{thm:CGstabilizing} provides the decomposition $\beta_K=\beta_1 \oplus \beta_2$, where~$\beta_1$ is presented by a rank $4$ matrix and where $\beta_2$ admits a metabolizer $G$ such that for every character~$\chi \in G$, the following inequality holds:
\begin{equation} \label{eq:GoalAlgSlice}
|\sigma(K,\chi)|-\eta(K,\chi) \leq 5.
\end{equation}
We compute the Casson-Gordon invariant using satellite formulas; cf. \cite{GilmerGenus,FlorensGilmer} for an approach using surgery formulas.
Let $F$ denote the Seifert surface for~$K$ given by the disk-band form depicted in Figure~\ref{fig:CExample}, and let $e_1,f_1,e_2,f_2,e_3,f_3$ be the curves in $S^3 \setminus F$ that are Alexander dual to the canonical generators of $H_1(F;\Z)$, i.e. to the cores of the bands of $F$.
Recall that these curves generate $H_1(\Sigma_2(R);\Z)$.
Since $H_1(\Sigma_2(K);\Z) =\oplus_{i=1}^3 H_1(\Sigma_2(R);\Z)$, we write characters as $\chi=\chi_1 \oplus \chi_2 \oplus \chi_3$. 
Define $\delta$ to be the number of non-trivial characters among $\chi_1,\chi_2,\chi_3$ 
minus one. 
Thus $0 \leq \delta \leq 2$ for $\chi \neq 0$.
Set $\omega:=e^{2\pi i/3}$. Applying the satellite formulas of Theorem~\ref{thm:AbchirSimplified}, we obtain
\begin{align*}
 \sigma(K,\chi) &=\sum_{i=1}^3 \big( \sigma(R,\chi_i)+2\sigma_{J}\big(\omega^{\chi_i(e_i)}\big)+2\sigma_{J}\big(\omega^{\chi_i(f_i)}\big) \big), \\
  \eta(K,\chi) &=\delta+\sum_{i=1}^3 \big( \eta(R,\chi_i)+2\eta_{J}\big(\omega^{\chi_i(e_i)}\big)+2\eta_{J}\big(\omega^{\chi_i(f_i)}\big) \big)
	= \delta+\sum_{i=1}^3 \eta(R,\chi_i) \nonumber,
\end{align*}
where in the last equality we used that $\eta_J(\omega) = 0$ if $\Delta_J(\omega) \neq 0$. 
To see this latter fact, note that $\eta_J(\omega)$ can be defined as the nullity of the Hermitian matrix $(1-\omega)A+(1-\overline{\omega})A^T$, where $A$ is any Seifert matrix for $J$. 
If $\chi(e_i), \chi(f_i)$ do not all vanish, then 
\begin{align*}
	|\sigma(K,\chi)|-\eta(K,\chi) &\geq 2 \Big | \sigma_J\big(\omega^{\chi_i(e_i)}\big) + \sigma_{J}\big(\omega^{\chi_i(f_i)}\big) \Big| - \Big| \sum_{i=1}^3 \sigma\big(R,\chi_i\big) \Big| - \eta(K,\chi)\\ 
	&\geq 2 \Big| \sigma_J(e^{2\pi i /3}) \Big| -  3\Big|\sigma(R, \wh \chi)\Big| -3\Big| \eta(R, \wh \chi') \Big| - \delta > 7-\delta \geq 5,
\end{align*}
for suitable choices $\wh \chi, \wh \chi' \colon H_1(\Sigma_2(R);\Z) \to \Z_3$.
Thus, if $G$ is non-trivial, then \eqref{eq:GoalAlgSlice} cannot be fulfilled, and $\sn(K) >1$.

The same reasoning shows that $g_4^{\op{top}}(K)=3$. 
Namely, if $g_4^{\op{top}}(K)=2$, then Theorem~\ref{thm:CGGenus}
shows that $\beta_K=\beta_1 \oplus \beta_2$, with $\beta_1$ admitting a rank $4$ presentation matrix; the reasoning is analogous to the above. 
\end{example}

\subsection{The stabilizing number is not the \texorpdfstring{$4$--genus}{4-genus}.}\label{sub:NotGenus}
We give an example of an infinite family of knots~$R(J_1,J_2,J_3)$ with~$\sn(R(J_1,J_2,J_3)) =1$, but topological $4$--genus~$g_4^{\op{top}}(R(J_1,J_2,J_3))=2$, proving Proposition~\ref{prop:Genus2stabilizingNumber1} from the introduction.
The first part of this subsection is devoted to showing that $\sn(R(J_1,J_2,J_3)) \leq 1$, while the second uses Casson-Gordon invariants to show that $g_4^{\op{top}}(R(J_1,J_2,J_3))=2$ for appropriate choices of $J_1,J_2,J_3$.
Since $\sn(K)=0$ if and only if $K$ is slice, this also implies that $\sn(R(J_1,J_2,J_3))=1$.

\medbreak
\begin{figure}[!htb]
\includegraphics{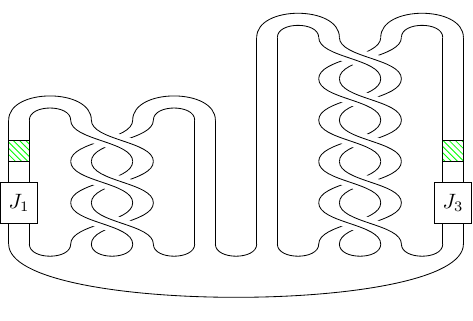}
\caption{The slice knot~$R(J_1, U, J_3)=R_1(J_1,U) \# R_2(U,J_3)$ depicted with the necessary band-moves.}
\label{fig:BandsToCut}
\end{figure}

Consider the knot $R(J_1, J_2, J_3)$ depicted in Figure~\ref{fig:CExample}. 
Observe that it is obtained as a winding number $0$ satellite with pattern $R:=R_1 \# R_2$ by infecting along the curves ${\gamma}_1,{\gamma}_2,{\gamma}_3$ also depicted in Figure~\ref{fig:CExample}. 
Since the ${\gamma}_i$ are winding number~$0$ infection curves, the following corollary is an immediate consequence of Lemma~\ref{lem:WindingPattern}.

\begin{corollary}
The knot~$K = R(J_1, J_2, J_3)$ has stabilizing number~$\sn(K) \leq 1$ for any choice of knots~$J_1, J_2, J_3$.
\end{corollary}
\begin{proof}
By Lemma~\ref{lem:WindingPattern}, the knot~$K$ is $1$--stably concordant to the knot~$R(J_1, U, J_3)$. The knot~$R(J_1, U, J_3)$ is slice by the bands-moves depicted in Figure~\ref{fig:BandsToCut}. Gluing the annulus to the slice disk, we obtain a stable slice disk for~$R(J_1, J_2, J_3)$ in~$D^4 \csum S^2 \times~S^2$.
\end{proof}

Next, we use Casson-Gordon invariants and Theorem~\ref{thm:CGGenus} to show that for appropriate choices of~$J_1,J_2,J_3$, the knot $R(J_1,J_2,J_3)$ has $4$--genus~$2$; note that a glance at Figure~\ref{fig:CExample} shows that $g_4^{\op{top}}(R(J_1,J_2,J_3))\leq 2$.
This will be based on a computation of the Casson-Gordon $\sigma$--invariant for satellite knots and will make use of the satellite formulas of Theorem~\ref{thm:AbchirSimplified}.

A symmetrized Seifert matrix for $R(J_1,J_2,J_3)$ is given~by 
\[ \begin{pmatrix} 0&3\\3&0 \end{pmatrix} \oplus  \begin{pmatrix} 0&5\\5&0 \end{pmatrix}. \]
Using Remark~\ref{rem:LinkingFormSeifert}, this matrix presents~$H_1\big(\Sigma_2(R(J_1,J_2,J_3);\Z\big) \cong \Z_3^2 \oplus \Z_5^2$ and the linking form it supports.

The next proposition provides a criterion ensuring that $g_4^{\op{top}}(R(J_1,J_2,J_3))=2$, which is fulfilled for example by taking $J_i$ to be a large enough connect sum of the knot~$-9_{35}$.
\begin{proposition}\label{prop:Genus2Stab1}
Assume the knots $J_1,J_2,J_3$ satisfy the following conditions for all~$\omega$ with $\omega^{15} = 1$:
\begin{enumerate}
\item $\eta_{J_i}(\omega)=0$, 
\item $\sigma_{J_i}(\omega) \geq \max_{\chi_1, \chi_2} \big\{ |\sigma(R_1,\chi_1)| + |\eta(R_1, \chi_1)|,|\sigma(R_2,\chi_2)| + |\eta(R_2, \chi_2)| \big\}+ 3$,
\end{enumerate}
where $\chi_1,\chi_2$ range over prime-power order characters~$H_1(\Sigma_2 (R_i); \Z) \to \Z_{15}$ for $i=~1,2$.
Then the knot $R(J_1,J_2,J_3)$ has topological $4$--genus at least $2$:
\[g_4^{\op{top}}(R(J_1,J_2,J_3)) \geq 2.\]
\end{proposition}
\begin{proof}
As we mentioned above, our goal is to apply the genus obstruction of Theorem~\ref{thm:CGGenus}.
Since $R(J_1,J_2,J_3)$ and $R(U,U,U) = R_1 \csum R_2$ have the same Seifert matrix,
we have $H_1(\Sigma_2(R(J_1,J_2,J_3));\Z) \cong H_1(\Sigma_2(R_1 \csum R_2);\Z)$ and the linking forms of the two are isometric. 

Abbreviate $R(J_1,J_2,J_3)$ by $K$. Pick a non-trivial prime-power order character $\chi \colon H_1(\Sigma_2(K);\Z) \to \Z_{15}$. In order to apply Theorem~\ref{thm:CGGenus}, we must compute the Casson-Gordon invariants $\sigma(K,\chi)$ and $\eta(K,\chi)$. This will be done by using the satellite formula described in Theorem~\ref{thm:AbchirSimplified}.
Observe that~$K$ is a winding number $0$ satellite knot with pattern $R(J_1,U,J_3)=R_1(J_1,U) \csum R_2(U,J_3)$, infection curve ${\gamma}_2$ and companion~$J_2$. 
The previous paragraph implies that to the character~$\chi \colon H_1(\Sigma_2(K);\Z) \to \Z_{15}$ correspond characters~$\chi_1 \colon H_1(\Sigma_2(R_1);\Z) \to \Z_{15}$ and~$\chi_2 \colon H_1(\Sigma_2(R_2);\Z) \to \Z_{15}$, which take values in $\Z_3 = 5 \cdot \Z_{15}$ and $\Z_5 = 3 \cdot \Z_{15}$. 
Since~$\chi$ has prime power order, one of the characters $\chi_1$, $\chi_2$ vanishes.

Let $\omega = e^{2\pi i/15}$. 
Using repeatedly the satellite and connected sum formulas of Theorem~\ref{thm:AbchirSimplified}, and remembering that one of the $\chi_i$ must be trivial,
we obtain
\begin{align*}
\sigma(K,\chi)
&=\sigma\big(R(J_1,U,J_3),\chi \big)+2\sigma_{J_2}(\omega^{\chi({\gamma}_2)}) \\
&=\sigma\big(R_1(J_1,U),\chi_1 \big)+\sigma\big(R_2(U,J_3),\chi_2\big)+2\sigma_{J_2}(\omega^{\chi({\gamma}_2)}) \\
&=\sigma(R_1,\chi_1)+\sigma(R_2,\chi_2)+2\sigma_{J_1}(\omega^{\chi_1({\gamma}_1)})+2\sigma_{J_3}(\omega^{\chi_2({\gamma}_3)})+2\sigma_{J_2}(\omega^{\chi({\gamma}_2)}), \\
\eta(K,\chi)&=\eta(R_1,\chi_1)+\eta(R_2,\chi_2)+2\eta_{J_1}(\omega^{\chi_1({\gamma}_1)})+2\eta_{J_3}(\omega^{\chi_2({\gamma}_3)})+2\eta_{J_2}(\omega^{\chi({\gamma}_2)}).
\end{align*}
We will now make these expressions more explicit. Let $F$ be the Seifert surface for $R_1 \csum R_2$ given by the disk-band form depicted in Figure~\ref{fig:BandsToCut}.
Use $e_1,f_1,e_2,f_2$ to denote the curves in $S^3 \setminus F$ that are Alexander dual to the canonical generators of $H_1(F;\Z)$. 
Observe that these curves are generators of~$H_1(\Sigma_2(K);\Z)$ and that ${\gamma}_1=e_1$, ${\gamma}_3=f_2$ and ${\gamma}_2=f_1-e_2$. As a consequence, we have 
\begin{align}\label{eq:SigmaRJ1J2J3}
 \sigma(K,\chi) &=\sigma(R_1,\chi_1)+\sigma(R_2,\chi_2) \\
&\quad+2\sigma_{J_1}(\omega^{\chi_1(e_1)})+2\sigma_{J_3}(\omega^{\chi_2(f_2)})
 +2\sigma_{J_2}(\omega^{\chi_1(f_1)} \overline{\omega}^{\chi_2(e_2)}), \nonumber\\
  \eta(K,\chi) &=\eta(R_1,\chi_1)+\eta(R_2,\chi_2), \nonumber
\end{align}
where we used that $\eta_J(\omega^k)$ is zero for $k = 0, \ldots, 14$.
\begin{claim}
If $\chi$ is a non-trivial prime-power order character, then \[ |\sigma(K,\chi)|-\eta(K,\chi)> 5. \]
\end{claim}
Since $\chi$ is non-trivial, at least one of $\chi_1(e_1)$, $\chi_1(f_1)$,$\chi_2(e_2)$, $\chi_2(f_2)$ is non-trivial.
Also, only one of $\chi_1(f_1)$, $\chi_2(e_2)$ can be non-trivial, since $\chi$ is of prime power order.
This implies that not all of
$\omega^{\chi_1(e_1)}$, $\omega^{\chi_2(f_2)}$, $\omega^{\chi_1(f_1)-\chi_2(e_2)}$ can be $1$. 
Using the hypothesis on $J_1,J_2,J_3$, for some $i \in \lbrace 1,2,3 \rbrace$ we obtain the estimate 
\begin{align*}
|\sigma(K,\chi)|-\eta(K,\chi)
&= \big|\sigma(R,\chi_1)+\sigma(R,\chi_2)+ 2\sigma_{J_1}(\omega^{\chi_1(e_1)})\\
&\quad+2\sigma_{J_3}(\omega^{\chi_2(f_2)}) +
 2\sigma_{J_2}(\omega^{\chi_1(f_1)} \overline{\omega}^{\chi_2(e_2)})\big| -\eta(K,\chi)\\
&\geq \sigma(R_1,\chi_1)+\sigma(R_2,\chi_2)+2\sigma_{J_i}(\omega^k) - \eta(R_1,\chi_1)-\eta(R_2,\chi_2) \\
& >  5
\end{align*}
for a suitable $k \in \Z$ which is not divisible by~$15$.

Now we use these observations to conclude.
 By way of contradiction, assume that~$K$ bounds a surface of genus $g=1$ in $D^4$. Theorem~\ref{thm:CGGenus} tells us that the linking form $\beta_K$ decomposes as~$\beta_1 \oplus \beta_2$, where $\beta_1$ has an even presentation matrix of rank $2g=2$ 
and $\beta_2$ has a metabolizer~$G$.
Thus $\beta_2$ is non-trivial, and so also $G$ is non-trivial. Deduce that $G$ has to contain a non-trivial element~$\chi$ of prime power order. Invoking Theorem~\ref{thm:CGGenus}, this element
has to fulfill the equation
\begin{equation}
 |\sigma(K,\chi)+0|-\eta(K,\chi)\leq 4g+1=5,
\end{equation}
which contradicts the claim above. We deduce that $g_4^{\op{top}}(K) \geq 2$.
\end{proof}

The next remark outlines how higher order invariants also give rise to lower bounds on the stabilizing number. We do not discuss the definition of these invariants but instead refer the interested reader to~\cite{CochranOrrTeichner} and~\cite[Section~2]{CochranHarveyLeidy}.

\begin{remark} \label{rem:HigherOrder}
Given a link $L$, we describe how von Neumann-Cheeger-Gromov $\rho$--invariants of the $0$--framed surgery $M_L$ produce lower bounds on $\op{sn}(L)$.
Assume that~$L$ is sliced by a nullhomologous disk~$D$ in $D^4 \csum n S^2 \times S^2$. 
If a group homomorphism $\phi \colon \pi_1(M_L) \to\Gamma $ factors through $\pi_1(D^4 \csum n S^2 \times S^2 \setminus D)$, then~\cite[Theorem 1.1]{ChaTopologicalMinimal} and an Euler characteristic computation imply that
\[ \rho(M_L,\phi) \leq 2 \beta_2(D^4 \csum n S^2 \times S^2 \setminus D)=4n.\]
The difficulty in applying this result lies in finding representations $\phi$ of $\pi_1(M_L)$ that factor through $\pi_1(D^4 \csum n S^2 \times S^2 \setminus D)$.
\end{remark}

\section{Stabilization and the \texorpdfstring{$4$--genus}{4-genus}}
\label{sec:4Genus}

We recall the relation between the Arf invariant of a knot~$K$ and framings on surfaces bounded by~$K$ using spin structures; see \cite{FreedmanKirby}, \cite[Section~XI.3]{Kirby89} and \cite[Section~11.4]{Scorpan05}. Then we surger a surface down to a disk while stabilizing the ambient manifold. 

\subsection{Stable framings and spin structures.}\label{sub:NotationBunlde}
In this subsection, we briefly recall the definition of the spin bordism group and fix some notations on vector bundles.
\medbreak

A (stable) \emph{spin structure} on a manifold~$M$ is a stable trivialization on the $1$--skeleton of the tangent bundle $TM$ that extends over its $2$--skeleton. 
Here, a spin structure refers to the stable notion, which is customary in bordism theory, and not the unstable one employed in gauge theory.
The bordism group of spin $n$--manifolds is denoted $\Omega_n^{\op{spin}}$. 
If the $n$--manifolds come with a map $f \colon M \to X$ to a fixed space~$X$, then we use $\Omega_n^{\op{spin}}(X)$ to denote the corresponding bordism group over $X$. 
We refer to~\cite[p.~16]{Stong} for details, but note that most of the literature uses the stable normal bundle of~$M$.
We will often specify spin structures on $M$ by indicating stable framings of $TM$.
\begin{remark}\label{rem:Omega2SpinFr}
Since a spin structure is a stable framing of the tangent bundle restricted to the $1$--skeleton that extends over the $2$--skeleton, on a surface it is just a stable framing. In other words, one has~$\Omega_2^{\op{spin}}=\Omega_2^{\op{fr}}$.
\end{remark}

Since bundles play an important role in this section, we start by fixing some notation and terminology.
\begin{remark}
If $\iota \colon \Sigma \looparrowright W$ is an immersion and $\xi$ is a bundle over $W$, then the \emph{restricted bundle} $\xi|_\Sigma$ over $\Sigma$ has the same fibers as $\xi$, but viewed over $\Sigma$. In other words, $\xi|_\Sigma$ is the pullback $\iota^*(\xi)$.
We will mostly consider the case where $\xi=TW$ is the tangent bundle. If we use $\nu_W(\Sigma)$ to denote the normal bundle of $\Sigma$ inside $W$, then we have 
\begin{equation}
\label{eq:RestrictedVersusNormal}
TW|_\Sigma=T\Sigma \oplus \nu_W(\Sigma).
\end{equation}
\end{remark}

Next, we discuss framings. Assume $\Sigma \subset W$ is a $k$--dimensional submanifold of a framed $m$--manifold~$W$. This means that we have fixed a trivialization of $TW$, that is, we have a framing given by sections~$(s_1,\ldots, s_m)$ and an isomorphism of vector bundles~$(s_1, \ldots, s_m) \colon \ul \R^m \xrightarrow{\sim} TW$ given by $((a_1, \ldots, a_m),x) \mapsto \sum_i a_i s_i(x) \in T_xW$.
Suppose in addition we are also given a framing~$(n_1, \ldots, n_{m-k}) \colon \ul \R^{m-k} \xrightarrow{\sim} \nu_W (\Sigma$). We obtain a stable tangential framing of $\Sigma$ by the composition 
\begin{equation}
\label{eq:FramingRestrictedBundle}
\op{fr}_\Sigma \colon   \ul \R^m \xrightarrow{(s_1, \ldots, s_m)} TW|_\Sigma \cong T\Sigma \oplus \nu_W (\Sigma) \xleftarrow{\id_{T\Sigma} \oplus (n_1, \ldots, n_{m-k})} T\Sigma \oplus \ul \R^{m-k}. 
\end{equation}
It is important to note that this stable framing of~$T\Sigma \oplus \ul \R^{m-k}$ depends not only on~$(s_1, \ldots, s_m)$ but also on the frame~$(n_1, \ldots, n_{m-k})$.

To keep the notation on the arrows at bay, we make use of the following notation.
\begin{notation}
\label{not:Sections}
Let $v_1, \ldots, v_m$ be sections of a vector bundle~$\xi$ over a manifold $M$. As above, define the map of vector bundles $\ul \R^m \rightarrow \xi$ by $((a_1, \ldots, a_m),x) \mapsto \sum_i a_i v_i(x)$. We will refer to this map by~$\ul \R^{m}\langle v_1, \ldots, v_m\rangle \rightarrow \xi$, and if this map is an isomorphism, then we write~$\ul \R^{m}\langle v_1, \ldots, v_m \rangle \xrightarrow{\sim} \xi$. 
\end{notation}

An oriented framing~$(s_1, \ldots, s_m)$ of an oriented vector bundle~$\xi$ defines a section of the oriented frame bundle~$\op{Fr}(\xi)$, whose fiber over a point~$x$ consists of all oriented bases of the fiber~$\xi_x$.
We say that two framings of~$\xi$ are \emph{homotopic}, if the associated sections in~$\op{Fr}(\xi)$ are homotopic through sections. If two stable framings~$f,f' \colon \ul \R^{m+k} \xrightarrow{\sim} TM \oplus \ul \R^k$ are homotopic, then the two framed manifolds~$(M,f)$ and $(M,f')$ are framed bordant via a suitable stable framing on~$M\times I$. 
Thus,~$[M,f] = [M,f'] \in \Omega^{\op{fr}}_n$.

For an $m$--dimensional oriented vector bundle over~$M$, the frame bundle~$\op{Fr}(\xi)$ of positively oriented frames is a~$\op{GL}_+(m)$--principal bundle. Consequently, given two framings~$s$ and $s'$, the equation~$s(x)\cdot h(x) = s'(x)$ for every~$x \in M$ defines a function~$h \colon M \to \op{GL}_+(m)$. 
Conversely, given a framing~$s$ and a function~$h \colon M \to \op{GL}_+(m)$, one can construct a new section~$s \cdot h$ by $(s\cdot h)(x) = s(x)\cdot h(x)$. 
Thus, the homotopy classes of sections on an $m$--dimensional bundle over $M$ are in bijection with~$[M, \op{GL}_+(m)] \cong [M, \op{SO}(m)]$.

\begin{remark}\label{rem:BordantHom}
We argue that on~$M=S^1$, two stable framings are bordant if and only if they are homotopic.
The stabilized Lie group framing~$\mathfrak{s}$ defines the only non-trivial element~$[S^1, \mathfrak{s}]$ in~$\Omega^{\op{fr}}_1 \cong \Z_2$; see e.g.~\cite[p.~521]{Scorpan05}.
On the other hand, the above discussion implies that framings of the stable tangent bundle~$TS^1 \oplus \ul \R^{k}$ correspond bijectively with ~$[S^1, \op{GL}_+(k+1)] \cong [S^1, \op{SO}(k+1)]$.
Since this group is isomorphic to~$\Z_2$ for $k \geq 2$, there are also only two homotopy classes of sections: the class of the Lie group  framing~$\mathfrak{s}$, and the nullbordant one.
\end{remark}

\subsection{The Arf invariant of a knot}
We describe equivalent definitions of the Arf invariant of a knot $K$.
The first uses a class determined by the $0$--framed surgery~$M_K$ in $\Omega_3^{\op{spin}}(S^1)$.
The second involves spin surfaces in $D^4$.
\medbreak
Given a knot $K \subset S^3$, our first aim is to define a stable tangential framing on~$M_K$.
Consider the standard embedding~$S^3 \subset \R^4$.
The normal bundle~$\nu_{\R^4} (S^3)$ 
is $1$--dimensional and oriented, so fix a framing~$\ul \R = \nu_{\R^4} (S^3)$. 
The coordinate vector fields therefore give 
a stable tangential framing
\begin{equation}\label{eq:StableS3}
\op{fr}_{S^3} \colon \ul \R^4  \sarr TS^3 \oplus \nu_{\R^4} (S^3) = TS^3 \oplus \ul \R.
\end{equation} 
By restriction, the stable framing $\op{fr}_{S^3}$ defines a stable tangential framing on the exterior $X_K=S^3 \setminus \nu(K)$. 
To extend the framing over the trace of the surgery and therefore obtain a framing on $M_K$, the unique framing on the 2-handle~$D^2 \times D^2$ has to agree on $S^1 \times D^2$ with the framing $\op{fr}_{S^3}$ on~$\nu(K) \cong K \times D^2$. 

\begin{construction}
Trivialize~$\nu(K) \cong K \times D^2$ using the Seifert framing.
This means that the curve~$K \times \{1\} \subset K \times D^2$ is the $0$--framed longitude of $K$, 
	or equivalently, frame the normal bundle~$\nu_{S^3}(K) \cong \ul \R^2 \langle t_S, n_S \rangle$ using the outer normal~$t_S$ and the normal vector field $n_S$ of a Seifert surface $S \subset S^3$~\cite[Proposition~4.5.8]{GompfStipsicz}.

This can be reformulated in terms of tangential framings as follows: the Seifert framing gives a stable framing 
\begin{align*}
\op{fr}_K \colon \ul \R^4 & \xrightarrow{\sim} T(K \times D^2)|_K \oplus \nu_{\R^4} (S^3)|_{K}
=TK \oplus \ul \R^2 \langle t_S,n_S \rangle \oplus \nu_{\R^4} (S^3)|_{K}.  \nonumber
\end{align*}
Since $TK \oplus \ul \R \langle t_S \rangle$ is $TS|_K$, the knot $K$ with the framing~$\op{fr}_K$ is bounded by the Seifert surface, and so is nullbordant.

Inside $D^2 \times D^2$, the circle $S^1 \times \{0\}$ has a stable tangential framing given by
	\[ \op{fr}_{S^1} \colon \ul \R^4 \xrightarrow{\sim} T(D^2 \times D^2)|_{S^1} = T(S^1 \times D^2)|_{S^1} \oplus \ul \R \langle t_{D^2}\rangle = TS^1 \oplus \ul \R^3,\]
where $t_{D^2}$ denotes the outer normal vector of $D^2$. By construction, $S^1$ bounds $D^2$, and so $\op{fr}_{S^1}$ is nullbordant as well.
By Remark~\ref{rem:BordantHom}, two stable tangential framings on $S^1$ are bordant if and only if they are homotopic.
This implies that the framings $\op{fr}_K$ and $\op{fr}_{S^1}$ give rise to a framing
\[\op{af} \colon \ul \R^4 \to TM_K \oplus \ul \R\]
on the $0$--surgery~$M_K = S^3 \setminus \nu(K) \cup_{K \times D^2} S^1 \times D^2 $, which is nullbordant.
\end{construction}
The $0$--surgery together with the framing~$\op{af}$ defines a class~$[M_K, \op{af}] \in \Omega_3^{\text{spin}}(S^1)$, where the map $M_K \to S^1$ classifies the abelianization $\pi_1(M_K) \to \Z$.
\begin{definition}\label{def:ArfKnot}
The \emph{Arf invariant} of~$K$ is $\op{Arf}(K) = [M_K, \op{af}] \in \Omega_3^{\text{spin}}(S^1) = \Z_2$.
\end{definition}

Our goal is to show that $\op{Arf}(K) $ can be computed from an arbitrary spanning surface $\Sigma \subset D^4$, provided it is endowed with an appropriate stable tangential framing. 
As a consequence, we first associate an Arf invariant to an arbitrary closed stably framed surface:
given 
such a~$[\Sigma, f] \in~\Omega_2^\text{spin}$, we construct a quadratic enhancement~$\mu_f$ of the intersection form $\cdot$ of $\Sigma$; this quadratic enhancement is then used to define the Arf invariant of $[\Sigma, f]$.
\begin{construction}\label{const:QuadRef}
Let $\Sigma$ be a closed surface with a stable framing $f \colon \ul \R^4 \to T\Sigma \oplus \ul \R^2$.
Let $\gamma \subset \Sigma$ be an embedded loop with normal vector~$n_\gamma$. 
Using the stable framing $f$ and $T\Sigma=T\gamma \oplus \ul {\R} \langle n_\gamma \rangle,$
we obtain an induced stable framing~$f_\gamma$ on~$\gamma$:
\begin{equation}
\label{eq:FramingLoopInSigma}
f_\gamma \colon \ul \R^{4} \xrightarrow{f} T\Sigma \oplus \ul \R^2 = T\gamma \oplus \ul \R\langle n_\gamma \rangle \oplus \ul \R^2.
\end{equation}
We associate to~$\gamma$ the element~$[\gamma, f_\gamma] \in \Omega_1^{\text{spin}} \cong \Z_2$. 
The resulting element only depends on the homology class of~$\gamma$ and the spin structure on $\Sigma$, which is the trivialization~$f$.
Mapping the loop $\gamma $ to $[\gamma, f_\gamma]$ gives rise to a map~$\mu_{f} \colon H_1(\Sigma; \Z) \to \Z_2$. 
This defines a quadratic refinement of the intersection form~\cite[p.~514]{Scorpan05} meaning that $ \mu_{f}(x+y) - \mu_{f}(x) - \mu_{f}(y) \equiv  x\cdot y \mod 2$
 for any~$x,y \in H_1(\Sigma; \Z)$.

From $\mu_f$, we extract a number~$\op{Arf}(\mu_f) \in \Z_2$ by the following algebraic procedure applied to~$V = H_1(\Sigma; \Z_2)$: given a non-singular quadratic form $(V,\lambda,\mu)$ over $\Z_2$ with $\dim_{\Z_2}V=2n$, the \emph{Arf invariant} can be defined by picking a symplectic basis~$e_1,\ldots,e_n,f_1,\ldots,f_n$ for $(V,\lambda)$, that is~$\lambda(e_i,f_j)=\delta_{ij}$ and~$\lambda(e_i,e_j)=0=\lambda(f_i,f_j)$, and setting~$\operatorname{Arf}(\mu) := \sum_{i=1}^n\mu(e_i)\mu(f_i)$.
\end{construction}

By Remark~\ref{rem:Omega2SpinFr}, specifying a stable framing~$f$ on a surface $\Sigma$ is the same as equipping $\Sigma$ with a spin structure. The following lemma shows that $\op{Arf}(\mu_f)$ only depends on the bordism class~$[\Sigma,f] \in \Omega_2^{\text{spin}}$.
\begin{lemma}\label{lem:ArfFramed}
The Arf invariant defines an isomorphism $\Omega_2^{\op{spin}} \xrightarrow{\sim} \Z_2$, where a bordism class~$[\Sigma, f]$ is mapped to $\op{Arf}(\mu_f)$.
\end{lemma}
\begin{proof} See e.g.~\cite[p.~523]{Scorpan05}. \end{proof}

We say that a collection of 
embedded curves~$\{e_i, f_i\}_{i=1}^g$ in a surface~$\Sigma$ is a \emph{symplectic basis of curves} if all of the following conditions hold:
\begin{enumerate}
        \item $e_i$ is disjoint from both~$e_j$ and $f_j$ for every~$i \neq j$, 
        \item for every~$i$, the curve~$e_i$ intersects $f_i$ transversely
         in exactly one positive intersection point,
        \item $\langle [e_1], [f_1], \ldots, [e_g],[f_g]\rangle = H_1(\Sigma; \Z)$.
\end{enumerate}

\begin{lemma}\label{lem:CutSystem}
	If a closed surface $\Sigma$ admits a 
    stable framing~$f$ such that $\op{Arf}(\mu_f) =~0$, then~$\Sigma$ contains a symplectic basis of curves $\{e_i, f_i\}_{i=1}^{g}$  with~$\mu_{f}(e_i) = 0$ for each~$i$.
\end{lemma}
\begin{proof}
Consider~$H_1(\Sigma;\Z)$ together with its symplectic intersection form and the quadratic enhancement~$\mu_{f}$. 
We assert that $\op{Arf}(\mu_f) = 0$ implies the existence of a symplectic basis~$\{e_i, f_i\}_{i=1}^g$ of~$H_1(\Sigma;\Z)$ with~$\mu_f(e_i) =~0$.
A proof can be found in~\cite[p.502]{Scorpan05}, but we outline the main steps.
First, we can write
\[ \big( H_1(\Sigma; \Z), \lambda, \mu_f \big) \cong a H^+ \oplus b H^- \text{ and } \op{Arf}(\mu_f) = b \text{ mod 2},\]
where $H^+,H^-$ respectively denote the standard $2$--dimension hyperbolic form on~$H \cong \Z\langle e, f\rangle$ equipped with the quadratic refinements~$\mu^+,\mu^-$ fulfilling $\mu^+(e) = 0,\mu^-(f)=1$ and $\mu^-(e) = \mu^-(f) = 1$. 
Since $2H^- \cong 2H^+$ and $b$ is even (thanks to our assumption), we see that
$H_1(\Sigma; \Z) \cong g H^+$, and so we simply define~$e_i, f_i$ to be the pair~$e,f$ in the $i$--th summand~$H^+$, concluding the proof of the assertion.

Any symplectic basis for homology can be realized as a geometric basis of curves~\cite[Second proof of Theorem~6.4]{Farb12}.
\end{proof}

In the case of a locally flat surface $\Sigma \subset D^4$ with boundary a knot $K$, we construct a stable framing~$f$ on~$\Sigma$ such that $\op{Arf}([\wh \Sigma,f]) = \op{Arf}(K)$. Here $\wh \Sigma$ denotes the (abstract) surface obtained by capping off~$\Sigma$ by a disk and, as we explain below, any stable framing of $\Sigma$ extends to~$\wh \Sigma$.
Note that in general, the embedding~$\Sigma \subset D^4$ does not extend to an embedding~$\wh \Sigma \subset D^4$.

\begin{construction} \label{construction:ArfSpinSurface}
Recall that $H_1(D^4 \sm \Sigma; \Z) = \Z$ is generated by a meridian of~$\Sigma$.
Since~$S^1$ is an Eilenberg-MacLane space~$K(\Z,1)$, the correspondence 
\begin{equation} [D^4 \sm \Sigma,S^1] \xrightarrow{\sim} H^1(D^4 \sm \Sigma; \Z) \xrightarrow{\sim} \op{Hom}(H_1(D^4 \sm \Sigma;\Z), \Z) \label{eq:MfdM}
\end{equation}
associates to the homomorphism $ H_1(D^4 \sm \Sigma;\Z) \to \Z$ sending a meridian to~$1$, 
a map~$h \colon D^4 \sm \Sigma \to S^1$. 
Pick a trivialization $\iota \colon \Sigma \times D^2 \to \nu \Sigma$ such that the composition 
        \[ H_1(\Sigma \times \{1\}; \Z) \to H_1(\Sigma \times S^1; \Z) \xrightarrow{\iota_*} H_1(D^4 \sm \Sigma; \Z) \xrightarrow{h_*,\cong} \Z\] 
vanishes.
We assert that by a homotopy of~$h$ near $\Sigma$, we can arrange that~$h$ agrees with the composition
\[ \nu \Sigma \sm \Sigma \xrightarrow{\iota^{-1}}\Sigma \times (D^2 \sm \{0\}) \xrightarrow{\op{pr}_2} (D^2 \sm \{0\}) \to S^1,\]
where the second map is the projection onto the second factor, and the third the projection onto the angle. Since the two following maps agree
\[
        \begin{tikzcd}
                H_1(\nu \Sigma \sm \Sigma ;\Z) \ar[bend right=-19]{r}{h_*} \ar[bend right=20]{r}{h'_*} & H_1(S^1;\Z),
        \end{tikzcd}
\]
we can homotope~$h$ to agree with $h'$ in a neighbourhood of~$\Sigma$, proving our assertion.

Arrange for $h$ to be transverse at~$1$.
Since $h \simeq h'$ near $\Sigma$, we see that $h^{-1}(1) \cup~\Sigma$ is a compact $3$--manifold $M$.
Since the restriction $h|_{S^3 \sm K} \colon S^3 \setminus K \to S^1$ also maps a meridian to~$1$, we observe that~$h|_{S^3\sm K}^{-1}(1) \cup K = M \cap S^3 =: S$ is a surface bounded by $K$. Via a homotopy supported near $S^3$ and in $D^4 \sm \nu \Sigma$, we may assume that $S$ is in fact connected, and therefore a Seifert surface.

Let $n_M$ denote the normal vector of $M$ in $D^4$, which is a section of $\nu_{D^4}(M)$. Denote the outer-normal vector of~$M$ on $\Sigma$ by~$t_M$ (a section of $\nu_M(\Sigma)$). Applying~\eqref{eq:RestrictedVersusNormal} twice to the chain~$\Sigma \subset M \subset D^4$ of codimension~$1$ inclusions, we obtain a stable framing $f $ on $\Sigma$:
\[ f \colon \ul \R^4 \sarr TM|_\Sigma \oplus \ul \R\langle n_M \rangle = T\Sigma \oplus \ul \R^2\langle t_M, n_M\rangle,\]
\end{construction}

Any stable framing of $\Sigma$ 
restricts to the nullbordant stable framing on~$\partial \Sigma$.
Thus we can cap off the boundary component and obtain a stable tangent framing~$f$
on the closed surface~$\wh \Sigma$, which defines an element~$[\wh \Sigma, f] \in \Omega_2^\text{spin}$.
In particular, using Lemma~\ref{lem:ArfFramed}, it defines an Arf invariant.
\begin{proposition}\label{prop:ArfCoincide}
The Arf invariant of the element~$[\wh \Sigma, f] \in \Omega_2^\text{spin}$ coincides with the Arf invariant of~$K$.
\end{proposition}
\begin{proof}
Recall that we constructed the $3$--manifold $M \subset D^4$ as $M=h^{-1}(1) \cup \Sigma$, where the map $h \colon D^4 \sm \Sigma \to S^1$ was described in Construction~\ref{construction:ArfSpinSurface}. 
Recall furthermore that~$\partial M = \Sigma \cup_\partial S$, where~$S$ denotes the Seifert surface~$h|_{S^3 \sm K}^{-1}(1)$ of~$K$.

The stable framing $f_M \colon\ul \R^4 \sarr TM \oplus \ul \R$ induces the stable framing~$f$ on $\Sigma$ and a stable framing $-f_S$ on $S$. We deduce the following equation in $\Omega_2^{\op{spin}}$: 
\[ 0 =\partial [ M,  f_M] = [\Sigma \cup_\partial -S, f \cup -f_S].\]
Note that the induced stable tangent framing on the separating curve $\partial$ is nullbordant, and so extends over a disk. This allows us to perform surgery on~$\partial$. The result of this surgery is a disjoint union $\widehat{\Sigma} \sqcup \widehat{S}$ of closed surfaces obtained by capping off $\Sigma$ and $S$. 
Since the Arf invariant is a spin bordism invariant, we obtain:
\[ 0= [\Sigma \cup_\partial -S, f \cup -f_S]=\op{Arf}\Big(\Big[\wh \Sigma, \wh f\Big]\Big)-\op{Arf}\big(\big[\widehat{S},f_{\wh S}\big]\big).\]
It remains to argue that
$\op{Arf}([\widehat{S}, f_{\wh S}])=\op{Arf}(K)$.
The Atiyah-Hirzebruch spectral sequence provides an isomorphism~$\Phi \colon \Omega_3^\text{spin}(S^1) \xrightarrow{\sim} \Omega_2^{\op{spin}}$.

Given $x :=[Y,f, g] \in \Omega_3^\text{spin}(S^1)$ with $f$ a framing and~$g \colon Y \to S^1$, the closed surface underlying $\Phi(x)$ is $g^{-1}(\pt)$
where $\pt \in S^1$ is a point to which $g \colon M \to S^1$ is transverse.
If we return to our case,
this correspondence sends the class~$\big[M_K, f, \bar h\big]$ to~$\big[\widehat{S}, f_{\widehat{S}}\big]$, where $\bar h \colon M_K \to S^1$ is the obvious extension 
of~$h|_{S^3\sm \nu K}$ to the zero surgery~$M_K$. 
Now the preimage~$\bar h^{-1}(1)$ is exactly the capped of Seifert surface~$\widehat S$.
This concludes the proof of the lemma.
\end{proof}

\subsection{The stabilizing number and the \texorpdfstring{$4$--genus}{4-genus}}
Let $\gamma \subset D^4$ be a closed embedded curve. In order to perform surgery on $\gamma$, we must specify a framing of $\nu_{D^4}(\gamma) \in \Omega_1^{\op{fr}}$. We show that if this bordism class is zero, then the result of the surgery is $D^4 \csum S^2 \times S^2$.
\medbreak
Let $\gamma$ be curve on a spanning surface~$\Sigma \subset D^4$ for a knot~$K\subset S^3$.
A choice of a trivialization of the tubular neighborhood~$\nu( \gamma) \cong \gamma \times D^3$ gives rise to a stable framing of~$T\gamma$ by
$T (\gamma \times D^3) = TD^4|_{\gamma \times D^3}$.
In \eqref{eq:FramingLoopInSigma}, we constructed a stable framing~$f_\gamma$ of~$T \gamma$ by
\[ f_\gamma \colon \ul \R^4 \sarr T\gamma \oplus \ul \R^3\langle n_\gamma, n_M, t_M\rangle. \]
where $n_\gamma, n_M$ and $t_M$ are respectively sections of $\nu_\Sigma(\gamma), \nu_M(D^4)$ and $\nu_M(\Sigma)$. 
These vectors also give a framing of the normal bundle $\nu_{D^4}(\gamma)$.

Since we have a framed embedded circle in $D^4$, we can perform $1$--surgery on $\gamma$. We write $\op{surg}\big(D^4, \gamma, f_\gamma\big) = \op{surg}\big(D^4, \gamma, (n_\gamma, t_M, n_M) \big)$ for the effect of the surgery.

\begin{lemma}\label{lem:TrivialSurgery}
Let $\gamma \subset W$ be an embedded loop in a $4$--manifold~$W$. Assume that $\gamma$ is contained in a $4$--ball and let $\op{triv}_\gamma \colon \gamma \times D^3 \sarr \nu_{D^4} (\gamma)$ be a trivialization inducing the framing~$f_\gamma \colon \ul \R^4 \sarr T\gamma \oplus \ul \R^3$. If $[\gamma, f_\gamma] \in \Omega_1^{\text{fr}}$ is trivial, then the result of surgery along~$(\gamma, f_\gamma)$ is
\[  \op{surg}(W, \gamma, f_\gamma) = W \csum S^2 \times S^2.\]
\end{lemma}
\begin{proof}
Since the loop $\gamma$ is contained in a $4$--ball, we may assume that $\gamma$ is contained the~$S^4$ summand of a connected sum~$W \csum S^4$.
Consequently, it is enough to verify that~$\op{surg}(S^4, \gamma, f_\gamma) = S^2 \times S^2$. Isotope~$\gamma$ to the unit circle~$U \subset \R^2 \times \{0\} \subset \R^4 \cup \{\infty\}$. 
The nullbordant framing is
represented by the coordinate vectors~$f_\gamma = (e_1, \ldots, e_4)$. This framing is induced by the trivialization~$(n_U, e_3,e_4)$ of $\nu_{D^4}(U)$, where~$n_U$ is the normal vector of $U \subset \R^2$. 

This gives exactly the decomposition~$S^4 = S^1 \times D^3 \cup D^2 \times S^2$, and thus after replacing
~$S^1 \times D^3$ with $D^2 \times S^2$, we obtain~$S^2 \times S^2$.
\end{proof}

As we shall see below, performing surgery on half a symplectic basis of a genus~$g$ locally flat surface $\Sigma \subset D^4$ gives rise to a disk in $D^4 \csum g S^2 \times~S^2$. We must verify that this disk is nullhomologous.
\medbreak
Let $\Sigma \subset D^4$ be a properly embedded surface and let $\gamma \subset \Sigma$ be an embedded loop in $\Sigma$. Pick an embedded disk~$C \subset D^4$ with~$\partial C = \gamma$, which intersects~$\Sigma$ normally along the boundary~$\partial C$ and transversely in~$\Int C$; see Figure~\ref{fig:gammaBundle} below.
 Such a disk~$C$ is called a \emph{cap} for $\gamma \subset \Sigma$ if the algebraic intersection number~$\Sigma \cdot C$ is zero. In the literature, caps are usually neither assumed to be embedded nor to have winding number zero with $\Sigma$~\cite{CochranOrrTeichner, FreedmanQuinn, CST}.
 Since all our caps will have both these properties, we permit ourselves this shortcut.

\begin{remark}
A cap for $\gamma \subset \Sigma$ always exists: first, pick any spanning disk~$C \subset D^4$ for~$\gamma$, since $\gamma$ is unknotted in $4$--space such a disk exists. Now via an isotopy of~$C$ arrange that $C$ intersects $\Sigma$ normally in~$\gamma$. After a further isotopy supported in~$D^4 \sm \nu(\gamma)$, we may assume that~$\Int C$ intersects $\Sigma$ transversely.
        After these isotopies, the disk~$C$ will still be embedded. 
Now arrange that~$C \cdot \Sigma=0$, by spinning~$C$ around~$\partial C$~\cite[Figure 11.14]{Scorpan05}. The resulting embedded disk is the required~cap.
\end{remark}

Next, we use a cap to define a second stable framing on an embedded loop $\gamma \subset \Sigma$ and compare it to the stable framing $f_\gamma$ from Construction~\ref{const:QuadRef}.

\begin{construction}\label{construction:FramingFromCap}
Consider a cap~$e \colon D^2 \hookrightarrow D^4$, which is called~$C$. 
 Pick polar coordinates~$(r,\theta) \mapsto re^{i\theta}$ on~$D^2$ and consider the vector field~$\partial_r$ on $\partial D^2$; see Figure~\ref{fig:Outward}.
 
\begin{figure}[!htb]
\includegraphics{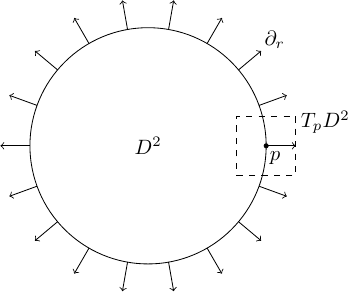}
\caption{The outward pointing vector field on $D^2$.}
\label{fig:Outward}
\end{figure} 
 
Since a cap~$C$ is in particular an embedding near the boundary,
         the push-forward~$e_*\partial_r$ defines a vector field~$t_C \in TC|_{\partial C} \subset TD^4|_{\partial C}$, which is tangential to~$C$ and which we call an \emph{outer-normal} vector field of~$C$; see Figure~\ref{fig:gammaBundle}.
As $C$ intersects $\Sigma$ normally along $\partial C$, we deduce that $t_C$ is also a section of $\nu_{D^4}(\Sigma)|_{\partial C}$.
As $\nu_{D^4}(\Sigma)|_{\partial C}$ is a $2$--dimensional oriented vector bundle, the section $t_C$ can be complemented with a linearly independent section~$v_\gamma(C)$ such that $(t_C, v_\gamma(C))$ is a positively oriented frame of $\nu_{D^4}(\Sigma)|_{\partial C}$. The section~$v_\gamma(C)$ is unique up to homotopy,
but note that it might not extend to a section of $\nu_{D^4}(C)$.
Denote the normal vector of~$\gamma$ in~$\Sigma$ by~$n_\gamma$.
Using~\eqref{eq:RestrictedVersusNormal} twice, we obtain a stable framing of~$T\gamma$~as
\begin{equation}\label{eq:capFraming}
	\op{fr}_\gamma(C) \colon \ul \R^4 \sarr TD^4 = T\Sigma|_\gamma \oplus \ul \R^2\langle t_C, v_\gamma(C)\rangle = T\gamma \oplus \ul \R^3\langle n_\gamma, t_C, v_\gamma(C)\rangle.
\end{equation}
\end{construction}

We wish to relate the stable framing $\op{fr}_\gamma(C)$ of $\gamma$ to the stable framing $f_\gamma$ that we defined in~\eqref{eq:FramingLoopInSigma}. Recall the notion of homotopic framings from Section~\ref{sub:NotationBunlde}.
The next lemma relates the stable tangential framings $f_\gamma$ and $\op{fr}_\gamma(C)$ of $\gamma$.

\begin{lemma}\label{lem:capFraming}
	Let $C$ be a cap for an embedded loop $\gamma \subset \Sigma$. The stable framing~$\op{fr}_\gamma(C)$ of~\eqref{eq:capFraming} is homotopic to the stable framing~$f_\gamma$ on~$\gamma$ defined in~\eqref{eq:FramingLoopInSigma}.
\end{lemma}
\begin{proof}
We first recall the definition of $f_\gamma$. 
Proceed as in the proof of Proposition~\ref{prop:ArfCoincide} to construct a $3$--manifold~$M \subset D^4$, whose boundary is $\Sigma \cup_\partial S$ for some Seifert surface~$S$ of~$K$. 
Recall that $n_\gamma, n_M$ and $t_M$ respectively denote sections of $\nu_\Sigma(\gamma), \nu_M(D^4)$ and $\nu_M(\Sigma)$. The stable framing~$f_\gamma$ was defined as:
\[ f_\gamma \colon \ul \R^4 \sarr T\gamma \oplus \ul \R^3\langle n_\gamma, t_M, n_M\rangle. \]

First, we show that the two vector fields~$t_M$ and $t_C$ are homotopic
as $1$--frames of~$\nu_{D^4}(\Sigma)|_\gamma$. 
Consider its disk bundle~$D(\nu_{D^4}(\Sigma)|_\gamma)$, whose total space is a solid torus $V$, and the circle bundle~$S(\nu_{D^4}(\Sigma)|_\gamma)$.
The section $t_C$ of the $S^1$--bundle~$S(\nu_{D^4}(\Sigma)|_\gamma)$ is of the form $g \cdot t_M$ for a suitable map $g \colon \gamma \to S^1$.
We will compute the homotopy class $[\gamma, S^1] \cong \Z$ of $g$.
So consider push-offs $\gamma^{t_M}$, $\gamma^{t_C}$ of $\gamma$ into the $t_M$ and $t_C$ direction. 
These are curves on the $2$--torus~$\partial V$, which are homotopic
to $\gamma$ in the solid torus $V$; they are \emph{longitudes} of $V$.
As in the case of knot, the homology class of longitudes can be recovered from the linking number.
Indeed, note that~$H_1( D^4 \sm \Sigma ; \Z) \cong \Z\langle \mu_\Sigma \rangle$ is generated by the meridian of $\Sigma$, and define the linking number~$\lk(\gamma^+, \Sigma)$ of a curve~$\gamma^+$ that is disjoint from $\Sigma$ by $[\gamma^+] = \lk(\gamma^+, \Sigma) [\mu_\Sigma]$. 
Note that $[\gamma^{g\cdot t_M}] = [\gamma^{t_M}] + \deg (g) [\mu_\Sigma]$.
Deduce that if $\lk(\gamma^{t_M}, \Sigma)$ and $\lk(\gamma^{t_C}, \Sigma)$ agree, then $\deg (g) = 0$, and a homotopy of $g$ to the constant map will produce a homotopy between
the sections $t_M$ and~$t_C$.
We now compute these linking numbers.

Recall from~\eqref{eq:MfdM}, that $M \sm \Sigma$ is the 
inverse image of a point under $h \colon D^4 \sm \Sigma \to~S^1$.
The definition above implies that $\lk(\gamma^{t_M}, \Sigma) [S^1]  = h_* (\gamma^{t_M})$.
Since $\gamma^{-t_M} \subset~M$, the push-off~$\gamma^{-t_M}$ has the property that $h_* (\gamma^{-t_M}) = 0$, and so $\lk (\gamma^{t_M}, \Sigma) = \lk(\gamma^{-t_M}, \Sigma) =~0$.

Now we compute~$\lk(\gamma^{t_C}, \Sigma)$. Once we remove the neighborhood~$\Sigma$ from~$D^4$, the disk~$C$ will be punctured. One boundary component will be~$\gamma^{-t_C}$, and there will be one extra boundary component for each interior intersection point of~$C$ with~$\Sigma$, and these extra boundary components will form meridians of~$\Sigma$. Since a cap~$C$ has $C \cdot \Sigma = 0$ and $\Sigma$ is connected, these meridians can be canceled: construct a surface~$C' \subset D^4 \sm \nu (\Sigma)$ by tubing pairs of meridians corresponding to intersection points of opposite signs together. 
The surface~$C'$ has a single boundary component~$\gamma^{-t_C}$, and we conclude that
\[ \lk(\gamma^{t_C}, \Sigma) = \lk(\gamma^{-t_C}, \Sigma) = C' \cdot \Sigma = 0. \]

Since the two linking numbers agree, the $1$--frames $t_C$ and $t_M$ are homotopic.
Because $\nu_{D^4}(\Sigma)|_\gamma$ is a $2$--dimensional oriented bundle, there is an essentially 
unique way to complement a vector to a (positive) frame, and so $(t_C,n_C)$ and $(t_M,n_M)$ are also homotopic as frames of $\nu_{D^4}(\Sigma)|_\gamma$.
\end{proof}

\begin{figure}
\includegraphics{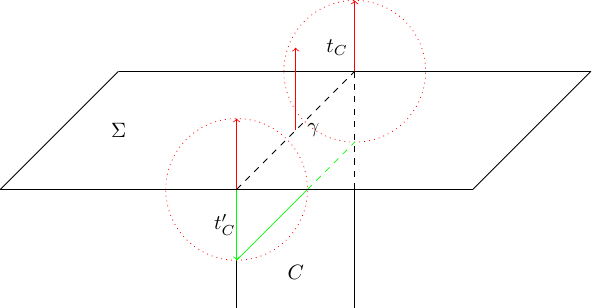}
\caption{A neighborhood of a point on $\gamma$. The dotted circles extend into the suppressed normal direction of $\Sigma$.}
\label{fig:gammaBundle}
\end{figure}

The next result is Theorem~\ref{thm:StablGenusIntro} from the introduction, whose proof is similar to \cite[Theorem 2]{FreedmanKirby}.
\begin{theorem}
\label{thm:StablGenus}
If $K$ is a knot with~$\op{Arf}(K) = 0$, then $sn(K) \leq g_4^{\op{top}}(K)$. 
\end{theorem}
\begin{proof}
Let~$\Sigma \subset D^4$ be a locally flat surface of genus~$g$.
Endow $\Sigma$ with the stable framing~$f$ described in Construction~\ref{construction:ArfSpinSurface}, so that $[\wh \Sigma,f]=\op{Arf}(K)$, as explained in Proposition~\ref{prop:ArfCoincide}.
Since $[\wh \Sigma,f]=\op{Arf}(K)=0$, Lemma~\ref{lem:CutSystem} implies that
$\Sigma$ contains a symplectic collection of curves $\lbrace \gamma_i,f_i \rbrace$ with $\mu_f(\gamma_i)=0$.
Pick a cap~$C_i$ for each curve~$\gamma_i$.

Note that the framing $(n_{\gamma_i},t_{C_i},v_{\gamma_i}(C_i))$ is a framing for $\gamma_i \subset D^4$, which restricts to a framing $n_{\gamma_i}$ for $\gamma_i \subset \Sigma$. Performing these (compatible) surgeries on $\gamma_i$ yields
\begin{align*} 
\Sigma' = \op{surg}\big(\Sigma, \gamma_i, (n_{\gamma_i})\big) &&\text{ and }&&
	W' = \op{surg}\big(D^4, \gamma_i,(n_{\gamma_i}, t_{C_i}, v_{\gamma_i}(C_i) )\big).
\end{align*}
Because the first vector of both framings is $n_{\gamma_i}$, the
surgered surface~$\Sigma'$ sits as a submanifold in~$W'$. The surface~$\Sigma'$ is a disk bounded by $K$.

Furthermore, the ambient manifold~$W'$ is~$D^4 \csum g S^2 \times S^2$:
since $\mu_f(\gamma_i) = 0$, the (normal) framing~$(n_{\gamma_i}, t_M, n_M)$ give rise to the nullbordant tangential framing of~$\gamma_i$. By Lemma~\ref{lem:capFraming}, $(n_{\gamma_i}, t_M, n_M)$ is 
homotopic to $(n_{\gamma_i}, t_{C_i}, v_{\gamma_i}(C_i) )$, and so the latter (normal) framing also gives rise to the nullbordant tangential framing of~$\gamma_i$. By Lemma~\ref{lem:TrivialSurgery}, the result~$W'$ of surgery along $(n_{\gamma_i}, t_{C_i}, v_{\gamma_i}(C_i) )$ is $D^4 \csum g S^2 \times S^2$. 

We have to check that the disk~$\Sigma'$ is nullhomologous in $W'=D^4 \csum g S^2 \times S^2$.
To show this, it is enough to show that $\Sigma'$ intersects algebraically zero with $2g$ linearly independent classes in $H_2(D^4 \csum g S^2 \times S^2;\Z)$.
We now construct these classes.
Surgery on $\gamma \subset D^4$ replaces $\gamma_i \times D^3$ with $D^2 \times S^2_i$.
The first $n$ linearly independent classes are $b_i:=[\{ \op{pt}_i \}\times S^2_i]$, where $\{ \op{pt}_i \} \in S^2_i$.

We verify that $\Sigma' \cdot b_i=0$ for $i=1,\ldots, g$. Note that $\Sigma'$ already intersects $\{ \op{pt}_i \} \times S^2_i$ transversely and:
\begin{equation}
\label{eq:IntersectLinIndep}
 \Sigma' \cdot \big( \{ \pt_i \} \times S^2_i \big) = 
\big( D^2 \times S^0 \big) \cdot \big( \{ \pt_i \} \times S^2_i \big) = S^0 \cdot S^2_i = 0. 
\end{equation}
We now construct $n$ more classes $a_i$ such that $a_i \cdot b_j=\delta_{ij}$ and $a_i \cdot \Sigma'=0$.
As we mentioned above, this is enough to show that $\Sigma'$ is nullhomologous.

The surgery along $\gamma_i$ is performed by removing $\gamma_i \times D^3$.
This $\gamma_i \times D^3$ was obtained by trivializing the normal bundle $\nu_{D^4}(\gamma_i)$ by $n_\gamma,t_{C_i},v_{\gamma_i}(C_i)$.
We think of~$n_\gamma,t_{C_i},v_{\gamma_i}(C_i)$ as an orthonormal system of coordinates for $D^3 \subset \R^3$. 
The surgery on $\gamma_i$ replaces $\gamma_i \times D_i^3$ by $D^2 \times S^2_i$.
Consider the $2$--disk $D_i^2:=D^2 \times \lbrace -t_{C_i} \rbrace \subset D_i^2 \times S^2_i $; this is an embedded disk in $W'$ whose boundary is the push-off~$\gamma_i^-$ of $\gamma_i$ in the $-t_{C_i}$ direction. Define $C_i^- := D^4 \setminus \Int \nu_{D^4}(\gamma) \cap C_i$, which is $C_i$ minus a boundary collar. The boundary~$\partial C_i^-$ is exactly $\gamma_i^-$. 
Consider the homology class 
\[a_i:=[ D^2_i \cup_{\gamma_i^-} C_i^-] \in H_2(W';\Z).\]

We argue that $a_i \cdot b_j=\delta_{ij}$ and $a_i \cdot \Sigma'=0$.
Note that $D_i^2$ meets $\lbrace \pt_i \rbrace \times S^2_i$ in a single point, namely $( \{ \pt_i \}, -t_{C_i})$ and so indeed $a_i \cdot b_j=\delta_{ij}$.
To compute the intersection~$a_i \cdot \Sigma'$, note that $\Sigma'$ is disjoint from $D^2_i$, and so $a_i \cdot \Sigma' = C \cdot \Sigma'$. Since~$C$ is a cap, it intersects $\Sigma$ zero algebraically, we get
\[ a_i \cdot \Sigma' =C \cdot \Sigma'= C \cdot \Sigma = 0. \]
As we have seen above, these equalities 
and~\eqref{eq:IntersectLinIndep} are enough to show that $\Sigma'$ is nullhomologous.
We have therefore constructed our nullhomologous slice disk for~$K$ in $D^4 \csum g S^2 \times S^2$, concluding the proof of the theorem.
\end{proof}

\appendix
\section{The satellite formula for the Casson-Gordon \texorpdfstring{$\sigma$}{sigma} invariant.} \label{Appendix}
As we reviewed in Subsection~\ref{sub:CGDef}, Casson and Gordon defined a signature defect~$\sigma(K,\chi)$ \cite{CassonGordon1,CassonGordon2} and a Witt class $\tau(K,\chi)$ \cite{CassonGordon1}.
Litherland proved a satellite formula for $\tau(K,\chi)$ \cite[Theorem 2]{Litherland} and Abchir proved a satellite formula for~$\sigma(K,\chi)$ \cite[Theorem 2, Case 2; Equation (4)]{Abchir}. While Litherland's formula  involves the signature of the companion knot, in some cases Abchir's formula does not.

The next example shows that something is missing from Abchir's formula.

\begin{example}\label{ex:Motivation}
Set $R:=9_{46}$ and consider the knot $K:=R(J_1,J_2)$ depicted in Figure~\ref{fig:KJ1J2} below. The $2$--fold branched cover of both $R$ and $K$ has first  homology~$\Z_3 \oplus \Z_3$. Set $\omega:=e^{2\pi i/3}$. 
Use Akbulut-Kirby's description~\cite{AkbulutKirby79} of the $2$--fold branched cover in terms of a surgery diagram.
For the $\Z_3$--valued character that maps each generator to $1$, this surgery diagram and the surgery formula for the $\sigma$--invariant~\cite[Theorem 6.7]{CimasoniFlorens} imply that
\[ \sigma(K,\chi)=\sigma(R,\chi)+2\sigma_{J_1}(\omega)+2\sigma_{J_2}(\omega).\]
On the other hand, Abchir's satellite formula~\cite[Theorem 2, Case 2]{Abchir} implies that $\sigma(K,\chi)=\sigma(R,\chi)$. The Levine-Tristram signature terms do not appear in this expression. 
\end{example}

\begin{figure}[!htb]
\includegraphics{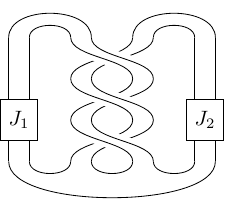}
	\caption{The satellite knot $R(J_1,J_2)$ described in Example~\ref{ex:Motivation}.}
\label{fig:KJ1J2}
\end{figure}

Theorem~\ref{thm:Abchir} below gives a corrected version of the formula, which agrees with the computation in Example~\ref{ex:Motivation}.
In Section~\ref{sub:BranchedSatellite}, we describe some constructions that appear in the proof of Theorem~\ref{thm:Abchir}; this proof is carried out in Section~\ref{sub:SatelliteProof}.
In Section~\ref{sub:NullitySatellite}, we prove a winding number zero satellite formula for the Casson-Gordon nullity.

\begin{notation}
	For a knot~$J$, we write $X_J$ for the exterior, $M_J$ for the $0$--framed surgery, and $\Sigma_n(J)$ for the $n$--fold branched cover of $J$. 
	The meridian $\mu_J$ of $J$ generates the groups $H_1(X_J; \Z) \cong \Z$ and $H_1(M_J; \Z) \cong \Z$, and the quotient homomorphisms~$H_1(X_J; \Z) \to \Z_n$ and $H_1(M_J; \Z) \to \Z_n$ give rise to $n$--fold covers denoted by $X_n(J)$ and $M_n(J)$.
While the longitude $\lambda_J$ of $J$ lifts to a loop in $X_n(J)$, the meridian~$\mu_J$ does not. 
	We use~$\widetilde{\mu}_J$ to denote the lift of~$\mu_J^n$.
	Finally, we use $\op{Ch}_{n}(J)$ to denote the set of characters $H_{1}(\Sigma_n(J);\Z) \to \Z_d$ for all $d$.
	 For a character $\chi \in \op{Ch}_{n}(J)$, we abbreviate $\sigma(\Sigma_n(J),\chi)$ by~$\sigma_n(J,\chi)$. 
\end{notation}

\subsection{Branched covers of satellite knots}
\label{sub:BranchedSatellite}
Let $P(K,{\gamma})$ be a satellite knot with pattern $P$, companion $K$, infection curve ${\gamma} \subset X_P$ and winding number~$w=~\ell k(P,{\gamma})$. Set $h:=\operatorname{gcd}(n,w)$ and use $\mu_{\gamma}$ for the meridian of ${\gamma}$ inside $X_P$. We recall the description of the branched cover $\Sigma_n(P(K,{\gamma}))$ in terms of $\Sigma_n(P)$ and~$X_{n/h}(K)$.
\medbreak
Consider the covering map~$\pi_n \colon X_n(P) \to X_{P}$. Use~$\phi \colon \pi_1(X_{P}) \to \Z$ for the abelianization map. Since $\phi({\gamma}) = w$ and the subgroup $w\cdot \Z_n \subset \Z_n$ has index~$h$, the set~$\pi_n^{-1}({\gamma})$ consists of $h$ components~${\gamma}_i$ for~$i=1,\ldots,h$. 
We refer to the~${\gamma}_i$ as \emph{lifts} of the infection curve ${\gamma}$ to the cover~$X_n(P)$.
Compared to $X_n(P)$, the submanifold $X_n(P) \sm \cup_{i=1}^h \nu ({\gamma}_i)$ has~$h$ additional boundary components~$\partial \nu ({\gamma}_i)$.
For each $i$, we frame $\partial \nu ({\gamma}_i)$ by lifting the framing of~$\partial \nu ({\gamma})$, that is we pick a circle~$\wt \mu^i_{\gamma} \subset \partial \nu ({\gamma}_i)$ covering~$\mu_{\gamma}$, and~$\wt \lambda^i_{\gamma} \subset \partial \nu ({\gamma}_i)$ covering~$\lambda_{\gamma}$.
Also, pick circles~$\wt \mu_K$ and $\wt \lambda_K$ in $\partial X_{n/h}(K)$ that cover the meridian and the longitude of $K$. 

Following Litherland~\cite[p.337]{Litherland}, the next lemma describes a decomposition of the branched cover $\Sigma_n(P(K,{\gamma}))$.

\begin{lemma}\label{lem:Decomposition}
	Let $P(K,{\gamma})$ be a satellite knot with pattern $P$, companion $K$, infection curve ${\gamma}$ and winding number $w=\ell k(P,{\gamma})$. Set $h:=\operatorname{gcd}(n,w)$.
	Take $h$ copies of $X_{n/h}(K)$, labeled by $X_{n/h}^i(K)$.
	Write $\wt \mu_K^i$ and $\wt \lambda_K^i$ for the
	curves~$\wt \mu_K$ and $\wt \lambda_K$ in the $i$--th copy $X_{n/h}^i(K)$, and ${\gamma}_1,\ldots,{\gamma}_h$ for lifts of ${\gamma}$ to $X_n(P)$.
	Then one has the following decomposition:
\[ \Sigma_n\big(P(K,{\gamma})\big)= \Big( \Sigma_n(P) \setminus \bigcup_{i=1}^h \nu({\gamma}_i) \Big) \cup \bigsqcup_{i=1}^h X_{n/h}^i(K), \]
where the identification $\partial \big( \Sigma_n(P) \setminus \bigcup_{i=1}^h \nu ({\gamma}_i) \big) \cong \bigsqcup_{i=1}^h \partial X_{n/h}^i(K)$ identifies $\wt \lambda^i_{\gamma} \sim \widetilde{\mu}^i_K$ and~$\wt \mu^i_{{\gamma}} \sim~\widetilde{\lambda}^i_K$.
\end{lemma}
\begin{proof} 
Observe that $X_{P(K,{\gamma})} = X_P \sm \nu({\gamma}) \cup X_K$, where $\partial \nu({\gamma})$ is glued
to $\partial \nu (K)$ 
via~$\mu_{\gamma} \sim~\lambda_K$ and $\lambda_{\gamma} \sim \mu_K$. 
Now consider the covering map~$X_n\big(P(K, {\gamma})\big)~\to~X_{P(K, {\gamma})}$. 
Note that $X_n(P(K,{\gamma}))$ contains a single copy of $X_n(P)$. The $h$ tori~$\partial \nu ( {\gamma}_i )$
 separate~$X_n(P)$ from the $h$ copies of $X_{n/h}(K)$, which we denoted by~$X_{n/h}^i(K)$.
The pieces~$X_n(P)$ and $X_{n/h}^i(K)$ are glued exactly as stated in the lemma.
This gives the following decomposition:
\[ X_n\big(P(K,{\gamma})\big)= \Big( X_n(P) \setminus \bigcup_{i=1}^h \nu({\gamma}_i) \Big) \cup \bigsqcup_{i=1}^h X_{n/h}^i(K).\]
To obtain the corresponding decomposition for $\Sigma_n\big(P(K,{\gamma})\big)$, fill the remaining boundary component with a solid torus. 
\end{proof}

\subsection{The satellite formula for the Casson-Gordon \texorpdfstring{$\sigma$--}{sigma }invariant.}
\label{sub:SatelliteProof}

Inspired by~\cite[proof of Lemma~2.3]{CochranHarveyLeidy}, we construct a cobordism $C$ between~$\Sigma_n(P(K, {\gamma}))$ and $\Sigma_n(P) \sqcup_i M^i_{n/h}(K)$ in order to relate the subsequent signature defects.
\begin{construction}\label{const:Cobordism} 
	Let $\lbrace{\gamma}_i \rbrace$ be the components of the preimage of ${\gamma}$ in $\Sigma_n(P)$; we still refer to them as \emph{lifts} of ${\gamma}$.
	Pick a tubular neighborhood~$\nu ({\gamma}_i) \cong S^1 \times D^2$ with meridian~$\wt \mu^i_{\gamma}$ and longitude~$\wt \lambda^i_{\gamma}$.  
	Since~$M_K$ is obtained as $X_K \cup S^1 \times D^2$ by identifying the meridian of the solid torus $S^1 \times D^2$ with $-\lambda_K$ and the longitude with $\mu_K$, the choice of the coefficient system on $M_K$ and $X_K$
	implies that the corresponding cover of $M_K$ is obtained as 
	\[  M_{n/h}(K) = X_{n/h}(K) \cup S^1 \times D^2, \]
where we glue in such a way that the meridian of $S^1 \times D^2$ is mapped to $-\wt \lambda_K$, and the longitude to $\wt \mu_K$.
Use $V_i \subset M_{n/h}^i(K)$ to denote the lift of this solid torus to~$M_{n/h}^i(K)$ for $i=1,\ldots,h$.
	The cobordism~$C$ is 
	defined by attaching a round~$1$--handle, i.e.\ as the quotient 
	\[ C := \Sigma_n(P) \times [0,1] \cup \bigcup_{i=1}^h \Big(M_{n/h}^i(K) \times [0,1]\Big)  \Big/ \sim,\]
	where the relation~$\sim$ identifies the solid torus $\nu({\gamma}_i) \times \{1\} \subset \Sigma_n(P) \times \{ 1 \}$ with the solid torus $V_i \times \{1\} \subset M_{n/h}^i(K)\times \lbrace 1 \rbrace$ via the diffeomorphisms~$\phi_i$ defined by $\phi_i\big(\wt \mu^i_{\gamma}\big) = \wt \lambda^i_K$ and $\phi_i\big({\gamma}_i) = \wt \mu^i_K$; here recall that the ${\gamma}_i$ are lifts of the infection curve ${\gamma}$.
            The bottom boundary of~$C$ is $\Sigma_n(P) \sqcup_{i}\, M_{n/h}^i(K)$ and the top boundary is~$\Sigma_n(P(K, {\gamma}))$ by Lemma~\ref{lem:Decomposition}; see Figure~\ref{fig:SchematicC}.
\begin{figure}
\includegraphics{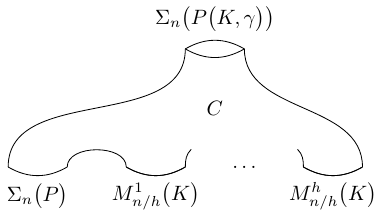}
\caption{A schematic of the cobordism~$C$.}
\label{fig:SchematicC}
\end{figure}

\end{construction}

Given a space $X$, write~$\op{Ch}(X)$ for the set of homomorphisms~$H_1(X;\Z)\to \Z_d$, and for a continuous map~$f \colon X \to Y$, write $f^* \colon \op{Ch}(Y) \to \op{Ch}(X)$ for the pullback. 
Note that if $f$ induces a surjection on $H_1$, then $f^*$ is injective.
Next, we describe the characters in $\op{Ch}(\partial C)$ that extend over $H_1(C;\Z)$.

Let $f_P\colon \Sigma_n(P) \sm \bigcup_i \nu({\gamma}_i) \hookrightarrow \Sigma_n(P(K, {\gamma}))$ be the inclusion induced by the decomposition of Lemma~\ref{lem:Decomposition}. 
Similarly, write~$f_{K,i} \colon X_{n/h}^i(K) \hookrightarrow \Sigma_n(P(K, {\gamma}))$. 
Also, write~$g_P \colon \Sigma_n(P) \sm \bigcup_i \nu({\gamma}_i) \hookrightarrow \Sigma_n(P)$ and~$g_{K,i} \colon X^i_{n/h}(K) \hookrightarrow M^i_{n/h}(K)$ for the 
evident inclusions. 

The next lemma describes the characters on $H_1(\partial C;\Z)$ that extend to the whole bordism~$C$; it also gives a description of the set~$\op{Ch}_n(P(K,{\gamma}))$; compare with~\cite[Lemma~4]{Litherland}. 
\begin{lemma} 
\label{lem:CorrespondenceCharacters}
	There exists a unique map~$\phi$ such that the diagram below commutes: 
	\[ 
	\begin{tikzcd}[row sep=1.3cm]
		\op{Ch}_n\big(P(K, {\gamma})\big) \ar[r,"\phi"] \ar{d}[swap]{\bsm f_P^{*} \\ f_{K,i}^{*} \esm}   & \op{Ch}_n(P) \oplus_i \op{Ch}(M^i_{n/h}(K)). \ar[dl, "g_P^{*} \oplus g_{K,i}^{*}"]\\
		\op{Ch}\big(\Sigma_n(P) \sm \bigcup_i \nu({\gamma}_i)\big) \oplus_i \op{Ch}(X^i_{n/h}(K))&
	\end{tikzcd}
	\]
	This map~$\phi$ has the following properties:
	\begin{enumerate}	
		\item the map~$\phi$ is injective and induces a bijection onto the set
			\[ E = \Big\{ (\chi_P,\{ \chi_i \} ) \in \op{Ch}_n(P) \oplus_i \op{Ch}(M^i_{n/h}(K)) \colon
			\chi_P({\gamma}_i) = \chi_i(\wt \mu^i_K ) \Big\}, \]
where ${\gamma}_1,\ldots,{\gamma}_h$ denote lifts of the infection curve ${\gamma}$;
		\item\label{lem:ExtensionCobordism} for $\phi(\chi) = (\chi_P, \{ \chi_i \})$, there exists a character~$\chi_C \colon H_1(C; \Z) \to \Z_d$ on the cobordism~$C$ of Construction~\ref{const:Cobordism} that restricts to $\chi$ on the top boundary and to $(\chi_P, \{\chi_i\})$ on the bottom boundary.
	\end{enumerate}
\end{lemma}
\begin{proof}
Mayer-Vietoris arguments show that $H_1(M_{n/h}^i(K);\Z)$ and $H_1(\Sigma_{n/h}^i(K);\Z)$ are obtained from $H_1(X_{n/h}^i(K);\Z)$ by respectively modding out the subgroup generated by $\widetilde{\lambda}_K^i$ and $\widetilde{\mu}_K^i$, while $H_1(\Sigma_n(P);\Z)$ is obtained from $H_1(\Sigma_n(P) \setminus \bigcup_i \nu({\gamma}_i) ;\Z)$ by modding out the subgroup generated by $\widetilde{\mu}_{\gamma}^i$.

The uniqueness of the map~$\phi$ follows from the fact that $g_P^*$ and $g_{K,i}^*$ are injective: indeed $g_P$ and $g_{K,i}$ induce surjections on $H_1$.

To establish the existence of $\phi$, we compute the image of~$g_P^*$ and $g_{K,i}^*$. 
The image of~$g_P^*$ is $\{ \chi_P \colon \chi_P( \wt \mu^i_{\gamma} ) = 0\}$.  
The image of~$g_{K,i}^*$ is $\{ \chi_i \colon \chi_i( \wt \lambda^i_K ) = 0\}$. 
What remains to be shown is that the range of 
$\bsm f_P^{*}, & f_{K,i}^{*} \esm^\intercal $
is contained in the sum of these images.
Suppose we are given a $\chi \in\op{Ch}_n\big(P(K, {\gamma})\big)$. 
Note that $\wt \lambda^i_K$ bounds a lift of a Seifert surface of $K$ in $X_{n/h}(K)$. Consequently,
$\chi( \wt \lambda^i_K ) = 0$. By Lemma~\ref{lem:Decomposition} we have $\wt \lambda^i_K \sim \wt \mu^i_{\gamma}$ in $\Sigma_n(P(K, {\gamma}))$ and so we deduce
that~$\chi(\mu^i_{\gamma}) = 0$. This establishes the existence of a map~$\phi$.

The injectivity of~$\phi$ follows from the fact that $f_P^* \oplus_i f_{K,i}^*$ is injective. As above, this follows from the fact that $f_P$ and $f_{K,i}$ induce surjective maps on $H_1$.

We check that $\phi$ surjects onto~$E$ and verify~\eqref{lem:ExtensionCobordism}. Let $(\chi_P,\{\chi_i\}) \in E$ be characters defined on the bottom boundary of the cobordism~$C$ of Construction~\ref{const:Cobordism}.
Since~$\chi_P ({\gamma}_i) = \chi_i(\wt \mu^i_K)$, these characters extend to a character~$\chi_C \colon H_1(C; \Z) \to \Z_d$ on all of $C$.
This proves~\eqref{lem:ExtensionCobordism}, but also produces a character~$\chi \in~\op{Ch}_n(P(K, {\gamma}))$ on the top boundary of $C$.
Since~$\big( \Sigma_n(P) \sm \bigcup_i \nu({\gamma}_i) \big)\times I \subset C$, we obtain that~$f_P^*(\chi) = \chi_P$. Similarly, $f_{K,i}^*(\chi) = \chi_i$. This shows that~$\phi(\chi) = (\chi_P,\{ \chi_i \})$, proving the surjectivity of $\phi$.
\end{proof}

The next result provides a satellite formula for $\sigma(K,\chi)$.
\begin{theorem}\label{thm:Abchir}
Let $n$ be a positive integer.
Let $P(K,{\gamma})$ be a satellite knot with pattern~$P$, companion~$K$, infection curve ${\gamma}$ and winding number $w$. 
Set $h:=~\op{gcd}(n,w)$.
Fix a character $\chi \in \op{Ch}_n(P(K,{\gamma}))$ of primer-power order and let $\chi_P \in \op{Ch}_n(P)$ and $\chi_i \in \op{Ch}(M_{n/h}^i(K))$ be the characters determined by the bijection of Lemma~\ref{lem:CorrespondenceCharacters}.
Define~$J:= \{ 1\leq i \leq h \colon \chi(\gamma_i) = 0\}$.
Then
\[ \Big|\sigma_n(P(K,{\gamma}),\chi) - \sigma_n(P,\chi_P)-\sum_{i=1}^h \sigma(M_{n/h}^i(K),\chi_i)\Big| \leq \#J,\]
where $\#J$ denotes the cardinality of~$J$.

Moreover, if the winding number~$w$ satisfies $w = 0$ mod~$n$, then 
\[ \sigma_n(P(K,{\gamma}),\chi) = \sigma_n(P,\chi_P) + \sum_{i=1}^h \sigma(M_K,\chi_i).\]
\end{theorem}
\begin{proof}
There is an $r>0$ and some (possibly disconnected)
$4$--manifolds $W(P)$ and~$W_1(K),\ldots,W_h(K)$ whose boundary respectively consist of the disjoint unions~$r \Sigma_n(P)$ and $r M_{n/h}^1(K),\ldots,r M_{n/h}^h(K)$, and such that the representations
$\chi_P$ and $\chi_1,\ldots,\chi_h$ respectively extend. 
Glue these $4$--manifolds to $r$ disjoint copies of the cobordism~$C$ of Construction~\ref{const:Cobordism} in order to obtain the (possibly disconnected) $4$-manifold
\[ W:= \bigsqcup_{j=1}^{r}C \cup \bigsqcup_{i=1}^h W_i(K) \cup W(P).\]
By construction, we have $\partial W=r\Sigma_n(P(K,{\gamma}))$.
Invoking the second item of Lemma~\ref{lem:CorrespondenceCharacters}, we know that $\chi=(\chi_P, \{ \chi_i \})$ extends to a character on $H_1(C;\Z)$.
The aforementioned characters therefore extend to a character on $W$.
Therefore,~$W$ can be used to compute $\sigma_n(P(K,{\gamma}),\chi)$.
Set $\dsign^\psi(C):=\sign^\psi(C)-\sign(C)$.
  Several applications of Wall's additivity theorem~\cite{WallAdditivity} imply that
\begin{equation}\label{eq:FormulaDsign}
 \sigma_n(P(K,{\gamma}),\chi)=r \dsign^\psi(C)+\sum_{i=1}^h \sigma(M_{n/h}^i(K),\chi_i)+\sigma_n(P,\chi_P). 
\end{equation}
The theorem will be proved once we show that $|\dsign^\psi(C)| \leq \#J$. 
By construction, the intersection of $\Sigma_n(P) \times [0,1]$ and $\bigsqcup_{i=1}^h M_{n/h}^i(K) \times [0,1]$ inside of $C$ consists of~$\bigsqcup_{i=1}^h {\gamma}_i \times D^2$; here recall that the ${\gamma}_i$ are lifts of the infection curve ${\gamma}$.
Consequently, the Mayer-Vietoris exact sequence for $C$ (where the coefficients are either~$\C^{\psi}$ or $\Z$) gives
\begin{align}\label{eq:MVForC}
\ldots &\to H_2(\Sigma_n(P) \times [0,1])  \oplus \bigoplus_{i=1}^h H_2(M_{n/h}^i(K) \times [0,1]) \to H_2(C) \to \bigoplus_{i=1}^h H_1({\gamma}_i \times D^2) \\
& \to H_1(\Sigma_n(P) \times [0,1]) \oplus \bigoplus_{i=1}^h H_1(M_{n/h}^i(K)  \times [0,1]) \to \ldots. \nonumber
\end{align}
We claim that with~$\Z$--coefficients the map $\iota \colon H_1({\gamma}_i \times D^2) \to H_1(M_{n/h}^i(K)  \times [0,1])$ is injective. 
Thanks to the identification ${\gamma}_i \sim \widetilde{\mu}_K^i$ and since $\widetilde{\mu}_K^i$ generates the $\Z$ summand of $H_1(M_{n/h}^i(K) ;\Z)=H_1(\Sigma_{n/h}^i(K);\Z) \oplus \Z$, 
the map $\iota$ sends ${\gamma}_i$ to $(0,1)$ in $H_1(M_{n/h}^i(K) ;\Z)$ and is therefore injective. This shows that the inclusion induced map~$H_2(\partial C; \Z) \to H_2(C; \Z)$ is surjective and the untwisted signature~$\sign C$ vanishes.

Next, we take care of the twisted case. For $i \notin J$, the coefficient system maps ${\gamma}_i$ to $\omega^{\chi({\gamma}_i)} \neq 1$, and so $H_1({\gamma}_i \times D^2;\C^{\psi})=0$. From~\eqref{eq:MVForC}, deduce that
\[ \dim_\C \op{coker} \big( H_2(\partial C; \C^\psi) \to H_2(C; \C^\psi) \big) \leq  \sum_{i \in J} \dim_\C H_1(\gamma_i \times D^2; \C^\psi).\]
This implies the inequality~$|\sign^\psi C| \leq \# J$. Deduce~$|\dsign^\psi(C)| \leq \#J$, and now the bound in Theorem~\ref{thm:Abchir} follows from~\eqref{eq:FormulaDsign}. This concludes the proof of the first part of the theorem.

Now consider the case where the winding number~$w$ satisfies $w = 0$ mod~$n$.
By assumption, we have $h=n$ and therefore $\sigma(M_{n/h}(K),\chi_i)=\sigma(M_K,\chi_i)$.
If~$\chi({\gamma}_i)$ is trivial, then the whole character~$\chi_i$ vanishes, and so~$H_1(\gamma_i \times D^2; \C^\psi) \to H_1(M_K; \C^\psi)$ is also injective. Thus, for all~$i$ the map~$H_1(\gamma_i \times D^2; \C^\psi) \to H_1(M_K; \C^\psi)$ is injective. As in the untwisted case above, deduce that~$\sign^\psi C = 0$ and so~$\dsign^\psi C = 0$. Since the signature defect vanishes, the case of a winding number~$0$ pattern follows from Equation~\eqref{eq:FormulaDsign}.
\end{proof}

We consider the case where the winding number is zero mod $n$.
\begin{corollary}\label{cor:Abchir0Modn}
Using the same notation as in Theorem~\ref{thm:Abchir}, we assume that the character $\chi$ is of prime power order $d$ and set $\omega:=e^{2\pi i/d}$.
 If the winding number~$w$ satisfies $w=0$ mod $n$, then 
\[ \sigma_n(P(K,{\gamma}),\chi)=\sigma_n(R,\chi)+\sum
_{i=1}^n \sigma_K(\omega^{\chi({\gamma}_i)}).\]
Here, recall that ${\gamma}_1,\ldots,{\gamma}_n$ denote lifts of the infection curve ${\gamma}$.
\end{corollary}
\begin{proof}
By assumption, we have $h=n$ and therefore $\sigma(M_{n/h}(K),\chi_i)=\sigma(M_K,\chi_i)$. By definition of $\chi_i$, we have $\chi_i(\mu_K^i)=\chi({\gamma}_i)$.
We claim that $\sigma(M_K,\chi_i)=\sigma_K(\omega^{\chi({\gamma}_i)})$.
If~$\chi({\gamma}_i)$ is trivial, then $\chi_i$ is trivial and so
$\sigma(M_K,\chi_i) = 0 = \sigma_K(1)$.
If~$\chi({\gamma}_i)$ is non-trivial, then the claim follows from the surgery formula for $\sigma(M,\chi)$~\cite[Lemma~3.1]{CassonGordon1}. 
Since $\sigma(M_K,\chi_i)=\sigma_K(\omega^{\chi({\gamma}_i)})$, the result now follows from Theorem~\ref{thm:Abchir}.
\end{proof}

The next example further restricts to the knot $R(J_1,J_2)$ described in Example~\ref{ex:Motivation}. 
The result agrees with the computation made in that example.
\begin{example} \label{ex:AbchirCorrected}
Set $\omega=e^{2\pi i/3}$. The knot $R(J_1,J_2)$ is obtained by $2$ successive satellite operations on $R$ along the curves ${\gamma}_1$, ${\gamma}_2$. 
	These curves generate the homology group $H_1\big(\Sigma_2(R(J_1,J_2));\Z\big)$ and we consider the $\Z_3$--valued character~$\chi$ that maps each of these curves to $1$.
Applying Corollary~\ref{cor:Abchir0Modn} a first time gives 
$\sigma(R(J_1,U))=\sigma(R,\chi)+2\sigma_{J_1}(\omega)$.
 Applying it a second time recovers the computation made in Example~\ref{ex:Motivation}:
\[\sigma(R(J_1,J_2),\chi)=\sigma(R,\chi)+2\sigma_{J_1}(\omega)+2\sigma_{J_2}(\omega).\]
\end{example}

\subsection{A satellite formula for the Casson-Gordon nullity}\label{sub:NullitySatellite}
Before describing a satellite formula for the Casson-Gordon nullity, we state an algebraic lemma whose proof is left to the reader.
\begin{lemma}\label{lem:HomolLem}
The following two statements hold:
\begin{enumerate}
	\item\label{item:HomolLem1} If $A \xrightarrow{\bsm f \\ g \esm } B \oplus C \xrightarrow{(h \ k)} D \to 0$ is exact and $g$ surjective, then
$\ker(g) \stackrel{f}{\to} B \stackrel{h}{\to} D \to 0$
is exact. 
\item\label{item:HomolLem2} If both sequences
$A \xrightarrow{\bsm f \\ 0 \esm} B\oplus C \to D \to 0 \text{ and } A \xrightarrow{f} B \to \widehat B \to 0$ are exact, then the map~$B\oplus C \to D$ descends to an isomorphism $\widehat B \oplus C \xrightarrow{\sim} D$.
\end{enumerate}
\end{lemma}

Given a character $\chi$, let $J$ denote the set~$\{1 \leq i \leq h \colon \chi({\gamma}_i) = 0\}$.
When $w=0$ mod $n$, we have $n=h$ and so $H_1(M_{n/h}(K);\Z)=H_1(M_K;\Z)=H_1(X_K;\Z)$, and Lemma~\ref{lem:CorrespondenceCharacters} associates to any character~$\chi \in \op{Ch}_n(P(K,{\gamma}))$
a character~$\chi_P \in \op{Ch}_n(P)$ and characters $\chi_i \in~\op{Ch}(X_K)$ for $i=1,\ldots,n$. 

The next result describes a satellite formula for the Casson-Gordon nullity.
\begin{proposition}\label{prop:CGNullitySatellite}
Let $n$ be a positive integer, let $P(K, {\gamma})$ be a winding number~$w\equiv 0 \text{ mod } n$
satellite knot and let $\chi \in \op{Ch}_n(P(K,{\gamma}))$ be a character of prime power order $d$.
 Set~$\omega:=e^{2 \pi i/d}$, and let~$(\chi_P, \{ \chi_i \})$ and $J$ be as above.  
Then
\[ \eta_n(P(K,{\gamma}), \chi) = \eta_n(P, \chi_P) +
	\sum_{i \notin J} \eta_K(\omega^{\chi({\gamma}_i)}).\]
Here, recall that ${\gamma}_1,\ldots,{\gamma}_n$ denote lifts of the infection curve ${\gamma}$.
\end{proposition}
\begin{proof}
Recall the decomposition~$\Sigma_n(P(K, {\gamma})) = \Sigma_n( P ) \sm \bigcup_i \nu ({\gamma}_i) \cup_i X_K^i$ that was described in Section~\ref{sub:BranchedSatellite}.
 Consider the following Mayer-Vietoris sequence with coefficients in~$\C^\chi$:
\begin{align}\label{eq:MV1}
&\to \bigoplus_i H_1(\partial \nu ({\gamma}_i)) \to H_1( \Sigma_n( P ) \sm \bigcup_i \nu ({\gamma}_i)) \oplus_i H_1(X_K^i) \to H_1(\Sigma_n(P(K, {\gamma}))) \to \phantom{0} \\
&\to \bigoplus_i H_0(\partial \nu ({\gamma}_i)) \to H_0( \Sigma_n( P ) \sm \bigcup_i \nu ({\gamma}_i)) \oplus_i H_0(X_K^i) \to H_0(\Sigma_n(P(K, {\gamma}))) \to 0. \nonumber
\end{align}
Since~$\chi$ is non-trivial, we deduce that $H_0(\Sigma_n(P(K, {\gamma}));\C^\chi) = 0$.
Given a connected space~$X$, recall that if $\psi \colon \pi_1(X \times S^1) \to \C$ is a homomorphism that is non-trivial on~$[\{\pt\} \times S^1]$, then~$C(X \times S^1;\C^{\psi})$ is acyclic.
In particular, we have $H_*(\partial \nu ({\gamma}_i))=0$ for~$i \notin J$.
Apply this to the sequence in~\eqref{eq:MV1} to obtain
\begin{align}\label{eq:MV2}
&\to \bigoplus_{i\in J} H_1(\partial \nu ({\gamma}_i)) \to H_1( \Sigma_n( P ) \sm \bigcup_i \nu ({\gamma}_i)) \oplus_i H_1(X_K^i) \to H_1(\Sigma_n(P(K, {\gamma})))\to \\
&\to \bigoplus_{i\in J} H_0(\partial \nu ({\gamma}_i)) \to H_0( \Sigma_n( P ) \sm \bigcup_i \nu ({\gamma}_i)) \oplus_{i \in J} H_0(X_K^i) \to 0. \nonumber
\end{align}
When $i \notin J$, we have $\chi({\gamma}_i) \neq 0$ and therefore
a Mayer-Vietoris argument also shows that we can fill back in the solid tori~$\nu ({\gamma}_i)$ without changing homology: namely one has $H_k( \Sigma_n( P ) \sm \bigcup_{i} \nu({\gamma}_i);\C^{\chi_P}) =  H_k( \Sigma_n( P ) \sm \bigcup_{i \in J} \nu ({\gamma}_i);\C^{\chi_P})$.

So far, we have simplified the Mayer-Vietoris sequence~\eqref{eq:MV2} to the sequence
\begin{align}\label{eq:MV3}
&\to \bigoplus_{i\in J} H_1(\partial \nu ({\gamma}_i)) \to 
\left\{ \begin{matrix}
H_1( \Sigma_n( P ) \sm \bigcup_{i\in J} \nu ({\gamma}_i))\\
\oplus_{i \notin J} H_1(X_K^i)\\
\oplus_{i \in J} H_1(X_K^i)
\end{matrix}\right\}
\to H_1(\Sigma_n(P(K, {\gamma})))\to\\
&\to \bigoplus_{i\in J} H_0(\partial \nu ({\gamma}_i)) \to 
	\left\{ \begin{matrix}
		H_0( \Sigma_n( P ) \sm \bigcup_{i \in J} \nu ({\gamma}_i))\\
		\oplus_{i \notin J} H_0(X_K^i) \\
		\oplus_{i \in J} H_0(X_K^i)
	\end{matrix}\right\}
\to 0.\nonumber
\end{align}
We focus on the degree~$0$ part of this sequence. 
For $i \in J$, we know that~$\chi_i(\mu^i_K)=~0$ and so the coefficient system on $X_K^i$ is trivial, i.e.\ $H_k(X_K^i;\C^{\chi_i})=H_k(X_K^i;\C)$;
the same observation holds for $\nu({\gamma}_i)$. Thus $\oplus_{i \in J} H_0(\nu({\gamma}_i);\C) \to \oplus_{i \in J} H_0(X_K^i;\C)$ is an isomorphism and the sequence~\eqref{eq:MV3} reduces to:
\begin{equation}\label{eq:MV4}
 \to \bigoplus_{i\in J} H_1(\partial \nu ({\gamma}_i)) \to 
\left\{ \begin{matrix}
H_1( \Sigma_n( P ) \sm \bigcup_{i\in J} \nu ({\gamma}_i))\\
\oplus_{i \notin J} H_1(X_K^i; \C^{\chi_i})\\
\oplus_{i \in J} H_1(X_K^i; \C)\\
\end{matrix}\right\}
\to H_1(\Sigma_n(P(K, {\gamma}))) \to 0.
\end{equation}

Recall that~$H_1(\partial \nu ({\gamma}_i);\C)=\C \langle \mu^i_K, \mu_{{\gamma}_i}\rangle$. Also, $\bigoplus_{i \in J} \C\langle \mu^i_K\rangle \xrightarrow{} H_1(X_K^i; \C)$ is an isomorphism. We apply Lemma~\ref{lem:HomolLem}~(\ref{item:HomolLem1}) to simplify the sequence~\eqref{eq:MV4} to 
\begin{equation}\label{eq:MV5}
  \bigoplus_{i\in J} \C\langle \mu_{{\gamma}_i}\rangle \to H_1( \Sigma_n( P ) \sm \bigcup_{i\in J} \nu ({\gamma}_i)) \oplus_{i \notin J} H_1(X_K^i; \C^{\chi_i}) \to H_1(\Sigma_n(P(K, {\gamma}))) \to 0.
\end{equation}

Before concluding, we must recover the homology of $\Sigma_n( P )$.
By filling the boundary tori of~$\Sigma_n( P ) \sm \bigcup_{i\in J} \nu ({\gamma}_i)$, the Mayer-Vietoris exact sequence and Lemma~\ref{lem:HomolLem}~(\ref{item:HomolLem1}) give
\begin{equation}\label{eq:FillTori} \to \bigoplus_{i\in J} \C\langle \mu_{{\gamma}_i}\rangle \to H_1( \Sigma_n( P ) \sm \bigcup_{i\in J} \nu ({\gamma}_i); \C^{\chi_P}) \to H_1(\Sigma_n(P); \C^{\chi_P}) \to 0.
\end{equation}
Applying Lemma~\ref{lem:HomolLem}~(\ref{item:HomolLem2}) to the sequences \eqref{eq:MV5} and \eqref{eq:FillTori} leads to the short exact sequence
\[ 0 \to H_1( \Sigma_n( P ); \C^{\chi_P}) \oplus_{i \notin J} H_1(X_K^i; \C^{\chi_i}) \to H_1(\Sigma_n(P(K, {\gamma})); \C^\chi) \to 0.\]
Consequently, $\eta_n(P(K, {\gamma}), \chi) = \eta_n(P, \chi_P) + \sum_{i \notin J}  \dim_\C H_1(X_K^i;\C^{\chi_i})$.
When $i \notin J$, we have~$\dim_\C H_1(X_K^i;\C^{\chi_i})=\eta_K(\omega^{\chi_i({\gamma}_i)})$ (see e.g. \cite[Proposition 3.4]{ConwayNagelToffoli}) and the proof is concluded.
\end{proof}

\bibliographystyle{alpha}
\bibliography{biblio}
\end{document}